\documentclass[11pt,a4paper]{amsart}
\usepackage[margin=2.8cm]{geometry}
\usepackage{amssymb}
\usepackage{graphicx} 
\usepackage{mathrsfs}
\usepackage{enumerate}
\usepackage{xspace}
\usepackage{color}
\usepackage{enumerate}
\usepackage{enumitem}
\usepackage{amsthm}
\usepackage{dsfont}
\usepackage{bbm}
\usepackage[colorlinks=true, linkcolor = blue, citecolor = blue]{hyperref}
\usepackage[numbers]{natbib}
\usepackage[backgroundcolor=white,bordercolor=red]{todonotes}
\DeclareMathAlphabet{\mathpzc}{OT1}{pzc}{m}{it}
\usepackage[nameinlink]{cleveref}
\numberwithin{equation}{section}
\begin{document}

\theoremstyle{plain}

\newtheorem{theorem}{Theorem}[section]
\newtheorem{lemma}[theorem]{Lemma}
\newtheorem{example}[theorem]{Example}
\newtheorem{proposition}[theorem]{Proposition}
\newtheorem{corollary}[theorem]{Corollary}
\newtheorem{definition}[theorem]{Definition}
\newtheorem{Ass}[theorem]{Assumption}
\newtheorem{condition}[theorem]{Condition}
\theoremstyle{definition}
\newtheorem{remark}[theorem]{Remark}
\newtheorem{SA}[theorem]{Standing Assumption}

\newcommand{\of}{[\hspace{-0.06cm}[}
\newcommand{\gs}{]\hspace{-0.06cm}]}

\newcommand\llambda{{\mathchoice
		{\lambda\mkern-4.5mu{\raisebox{.4ex}{\scriptsize$\backslash$}}}
		{\lambda\mkern-4.83mu{\raisebox{.4ex}{\scriptsize$\backslash$}}}
		{\lambda\mkern-4.5mu{\raisebox{.2ex}{\footnotesize$\scriptscriptstyle\backslash$}}}
		{\lambda\mkern-5.0mu{\raisebox{.2ex}{\tiny$\scriptscriptstyle\backslash$}}}}}

\newcommand{\1}{\mathds{1}}

\newcommand{\F}{\mathbf{F}}
\newcommand{\G}{\mathbf{G}}

\newcommand{\B}{\mathbf{B}}

\newcommand{\M}{\mathcal{M}}

\newcommand{\la}{\langle}
\newcommand{\ra}{\rangle}

\newcommand{\lle}{\langle\hspace{-0.085cm}\langle}
\newcommand{\rre}{\rangle\hspace{-0.085cm}\rangle}
\newcommand{\blle}{\Big\langle\hspace{-0.155cm}\Big\langle}
\newcommand{\brre}{\Big\rangle\hspace{-0.155cm}\Big\rangle}

\newcommand{\X}{\mathsf{X}}

\newcommand{\tr}{\operatorname{tr}}
\newcommand{\N}{{\mathbb{N}}}
\newcommand{\cadlag}{c\`adl\`ag }
\newcommand{\on}{\operatorname}
\newcommand{\oP}{\overline{P}}
\newcommand{\oO}{\mathcal{O}}
\newcommand{\D}{D(\mathbb{R}_+; \mathbb{R})}
\newcommand{\bx}{\mathsf{x}}
\newcommand{\bb}{\hat{b}}
\newcommand{\bs}{\hat{\sigma}}
\newcommand{\bk}{\hat{k}}
\renewcommand{\o}{\diamond}
\newcommand{\h}{R}

\renewcommand{\epsilon}{\varepsilon}

\newcommand{\fPs}{\mathfrak{P}_{\textup{sem}}}
\newcommand{\fPas}{\mathfrak{P}^{\textup{ac}}_{\textup{sem}}}
\newcommand{\rrarrow}{\twoheadrightarrow}
\newcommand{\cA}{\mathcal{C}}
\newcommand{\cR}{\mathcal{R}}
\newcommand{\cK}{\mathcal{K}}
\newcommand{\cQ}{\mathcal{Q}}
\newcommand{\cF}{\mathcal{F}}
\newcommand{\cE}{\mathcal{E}}
\newcommand{\cC}{\mathcal{C}}
\newcommand{\cD}{\mathcal{D}}
\newcommand{\cP}{\mathcal{P}}
\newcommand{\bC}{\mathbb{C}}
\newcommand{\cH}{\mathcal{H}}
\newcommand{\bth}{\overset{\leftarrow}\theta}
\renewcommand{\th}{\theta}
\newcommand{\cG}{\mathcal{G}}

\newcommand{\bR}{\mathbb{R}}
\newcommand{\bN}{\mathbb{N}}
\newcommand{\nnabla}{\nabla}
\newcommand{\f}{\mathfrak{f}}
\newcommand{\g}{\mathfrak{g}}
\newcommand{\oconv}{\overline{\on{co}}\hspace{0.075cm}}
\renewcommand{\a}{\mathfrak{a}}
\renewcommand{\b}{\mathfrak{b}}
\renewcommand{\d}{\mathsf{d}}
\newcommand{\bS}{\mathbb{S}^d_+}
\newcommand{\p}{\dot{\partial}}
\newcommand{\dr}{r} 
\newcommand{\m}{\mathbb{M}}
\newcommand{\Q}{Q}
\newcommand{\n}{\overline{\nu}} 
\newcommand{\usc}{\textit{USC}}
\newcommand{\lsc}{\textit{LSC}}
\newcommand{\q}{\mathfrak{q}}
\newcommand{\K}{\mathsf{M}}

\renewcommand{\emptyset}{\varnothing}

\allowdisplaybreaks

\makeatletter
\@namedef{subjclassname@2020}{%
	\textup{2020} Mathematics Subject Classification}
\makeatother

 \title[Nonlinear Semimartingales and Markov Processes with Jumps]{Nonlinear Semimartingales and Markov Processes \\ With Jumps} 
\author[D. Criens]{David Criens}
\author[L. Niemann]{Lars Niemann}
\address{Albert-Ludwigs University of Freiburg, Ernst-Zermelo-Str. 1, 79104 Freiburg, Germany}
\email{david.criens@stochastik.uni-freiburg.de}
\email{lars.niemann@stochastik.uni-freiburg.de}

\keywords{
nonlinear semimartingales; nonlinear Markov processes; sublinear expectation; sublinear semigroup; nonlinear expectation; partial differential equation; viscosity solution; semimartingale characteristics; Knightian uncertainty;
set-valued analysis; stochastic optimal control.}

\subjclass[2020]{47H20, 49J53, 49L25, 60G65, 60J60, 93E20}

\thanks{
LN acknowledges financial support from the DFG project SCHM 2160/13-1.}
\date{\today}

\maketitle

\begin{abstract}
In this paper we study a family of nonlinear (conditional) expectations that can be understood as a semimartingale with uncertain local characteristics. Here, the differential characteristics are prescribed by a time and path-dependent set-valued function. We show that the associated
control problem coincides with both its weak and relaxed counterparts.
Furthermore, we establish regularity properties of the value function and discuss their relation to Feller properties of nonlinear semigroups. In the Markovian case we provide conditions that allow us to identify the corresponding semigroup as the unique viscosity solution to a nonlinear Hamilton--Jacobi--Bellman equation. To illustrate our results we discuss a random \(G\)-double exponential L\'evy setting.
\end{abstract}


\section{Introduction}
A \emph{nonlinear Markov process} is a family of sublinear expectations \( \{\cE^x \colon x \in \bR^d\} \) on the Skorokhod space \( D(\bR_+; \bR^d) \), with
\( \cE^x \circ X_0^{-1} = \delta_x \) for each \( x \in \bR^d \), such that the Markov property
\begin{equation} \label{eq: markov property}
\cE^x(\cE^{X_t}(\psi (X_s))) = \cE^x(\psi (X_{t+s}) ), \quad x \in \bR^d, \ s,t \in \bR_+, 
\end{equation}
holds. Here, \(\psi\) runs through a collection of suitable test functions and \( X \) denotes the canonical process on \( D(\bR_+; \bR^d) \).
The study of nonlinear stochastic processes started with the seminal work of Peng \cite{peng2007g, peng2008multi} on the \(G\)-Brownian motion. More recently, larger classes of nonlinear (Markov) processes were constructed and analyzed in \cite{biagini, C22a, CN22a, CN22b, denk2020semigroup, fadina2019affine, hol16, hu2021g, K19, K21, NR, neufeld2017nonlinear, nutz}.
As in the theory of (linear) Markov processes, there is a strong link to semigroups. Indeed, the Markov property \eqref{eq: markov property} ensures the semigroup property \( S_t S_s = S_{s+t}, s,t \in \bR_+ \), where the sublinear operators \( S_t, t \in \bR_+, \) are defined by 
\begin{equation} \label{eq: def semigroup}
    S_t(\psi)(x) := \cE^x(\psi(X_t)). 
\end{equation}
Using the general theory of \cite{ElKa15, NVH}, the operators
\(S_t, t \in \mathbb{R}_+\), are well-defined on the cone of upper semianalytic functions. 
Further, a fundamental connection between the semigroup \((S_t)_{t \in \bR_+}\) and its (pointwise) generator
\(A\) was established in \cite{hol16}. More precisely, once the map \((t,x) \mapsto S_t(\psi)(x)\) is continuous, it solves the evolution equation
\begin{equation*} \label{eq: PDE A intro}
\begin{cases}   
\partial_t u (t, x) - A(u(t, \cdot \,))(x) = 0, & \text{for } (t, x) \in \bR_+ \times \mathbb{R}^d, \\
u (0, x) = \psi (x), & \text{for } x \in \bR^d,
\end{cases}
\end{equation*}
in a viscosity sense.
This necessitates the study of \emph{Feller properties} of the semigroup, i.e., whether it preserves regularity of the initial function \(\psi\).

For nonlinear L\'evy processes as introduced in \cite{neufeld2017nonlinear}, the corresponding semigroup has the Feller property on several spaces of continuous functions, which follow from the underlying additive structure \(\mathcal{E}^x (\psi (X_t)) = \mathcal{E}^0 (\psi (x + X_t))\), cf. \cite{denk2020semigroup,hol16, K19, K21}. 

For a more general class of nonlinear Markov semimartingales with jumps, the so-called \(\usc_b\)--Feller property was investigated in \cite{hol16}, i.e., it was shown that \(S_t\) is a self-map on the cone of bounded upper semicontinuous functions.\footnote{There appears to be a small gap in the proof of the \(\usc_b\)--Feller property, as the function \(\omega \mapsto \omega (t)\) is not upper semicontinuous in the Skorokhod \(J_1\) topology. This gap can be closed thanks to the fact that \(\omega \mapsto \omega (t)\) is almost surely Skorokhod \(J_1\) continuous under any law of a quasi-left continuous process.} The question whether the semigroup preserves lower semicontinuity was left as an open problem. 

In our recent papers \cite{CN22a,CN22b}, we addressed this problem in a general path-continuous Markovian and non-Markovian semimartingale framework. Several arguments developed in \cite{CN22a,CN22b} for the proof of the (strong) \(\usc_b\)--Feller property heavily rely on the path-continuous setting and appear not suitable in the presence of jumps.

In this paper, we continue this line of research and consider a general semimartingale framework that allows for jumps.
In particular, we do not restrict ourselves to Markovian situations but tackle the problem in a fully path-dependent framework. The role of the semigroup is played by the so-called {\em value function} that is given by 
\begin{align} \label{eq: value function intr}
v (t, \omega) := \mathcal{E}_t (\psi) (\omega) := \sup_{P \in \cA (t, \omega)} E^P \big[ \psi \big], \quad (t, \omega) \in \mathbb{R}_+ \times D(\mathbb{R}_+; \mathbb{R}^d), 
\end{align}
where \(\cC (t, \omega)\) is a set of probability measures on the Skorokhod space \(D (\mathbb{R}_+; \mathbb{R}^d)\) which give point mass to the path \(\omega\) till time \(t\) and afterwards coincide with the law of a semimartingale with absolutely continuous characteristics. 
We parameterize the differential characteristics by a compact parameter space \(F\) and three functions \(b \colon F \times \bR_+ \times D(\bR_+;\bR^d) \to \bR^d\), \(\sigma \colon F \times \bR_+ \times D(\bR_+;\bR^d) \to \bS\) and
\(k \colon F \times \bR_+ \times  D(\bR_+;\bR^d) \times L \to \bR^d \)
such that
\begin{align*}
\cC(t,\omega) := \Big\{ P \in \mathfrak{P}_{\text{sem}}^{\text{ac}}(t)\colon P&(X = \omega \text{ on } [0, t]) = 1, 
\\&(\llambda \otimes P)\text{-a.e. } (dB^P_{\cdot + t} /d\llambda, dC^P_{\cdot + t}/d\llambda, \nu^P_{\cdot + t} / d\llambda) \in \Theta (\cdot + t, X) \Big\},
\end{align*}
where \(L\) is a Lusin space equipped with a reference measure \(\n\), which can be seen as the state space of the randomness driving the jumps,
\[
\Theta (t, \omega) := \big\{(b (f, t, \omega), \sigma \sigma^* (f, t, \omega), \n \o k (f, t, \omega, \cdot)^{-1}) \colon f \in F \big\},
\]
and
\(\fPas(t)\) denotes the set of semimartingale laws after \(t\) with absolutely continuous characteristics.
In the Markovian case, where the set \(\Theta(t,\omega)\) depends on \((t, \omega)\) only through the value \(\omega(t-)\), the semigroup emerges
by setting
\begin{equation*} 
    S_t(\psi)(x) := \sup_{P \in \cC(0,x)} E^P \big[ \psi(X_t) \big].
\end{equation*}
The main results from this paper will establish general conditions on upper and lower semicontinuity of the value function \(v \colon [0, T] \times D([0, T]; \bR^d) \to \bR\), for a finite time horizon \(T > 0\), where \([0, T] \times D([0, T]; \bR^d)\) is endowed with the pseudometric \(\d\) that is given by
\begin{align*}\label{eq: intro qseudo metric}
\d ((t, \omega), (s, \alpha)) := |t - s| + \sup_{r \in [0, T]} \| \omega (r \wedge t) - \alpha (r \wedge s)\|
\end{align*}
for \((t, \omega), (s, \alpha) \in [0, T] \times D([0, T]; \bR^d)\).
To illustrate our findings, we discuss a fully path-dependent random \(G\)-double exponential setting. This class is related to a robust version of Kou's model (\cite{Kou}) and provides an extension of the random \(G\)-Brownian motion from~\cite{NVH} to a framework with jumps.

By virtue of the last term in \eqref{eq: value function intr}, it is a natural approach to deduce regularity properties of \(v\) with techniques from set-valued analysis. 
For controlled (jump) diffusions, this idea was successfully applied in the seminal paper \cite{nicole1987compactification} by El Karoui, Nguyen and Jeanblanc. 
The control framework considered in \cite{nicole1987compactification} is particularly well-suited for weak convergence techniques. 
In order to benefit from this, we embed our setting in a control environment. More specifically, we relate the value function to its weak and relaxed counterparts
  \[
  v^R (t, \omega) := \sup_{P \in \mathcal{P}^R (t, \omega)} E^P \big[ \psi (X) \big], \qquad v^W (t, \omega) := \sup_{P \in \mathcal{P}^W (t, \omega)} E^P \big[ \psi (X) \big],
  \]
where 
\(\mathcal{P}^R\) and \(\mathcal{P}^W \) denote
the set of relaxed and weak control rules, respectively.
Under modest assumptions on the coefficients, we establish the equalities
\[
v = v^R = v^W.
\]
In particular, we highlight that they connect the approaches from the seminal papers \cite{ElKa15} and \cite{neufeld2017nonlinear}.
Borrowing some arguments from \cite{nicole1987compactification}, we prove them with the help of Filippov's implicit function theorem and stability results for convex mixing. 

Our proof for the continuity of \(v\) is based on ideas from set-valued analysis that were used in \cite{nicole1987compactification} for the relaxed framework.
As a consequence of Berge's maximum theorem, upper and lower semicontinuity of the value function are closely linked to upper and lower hemicontinuity of \(\cP^R\) and \(\cP^W\), respectively. 
Upper hemicontinuity is characterized by a closed graph, which is easier to guarantee in the larger set \(\cP^R \supset \cP^W\).
In contrast, lower hemicontinuity requires the construction of an approximating sequence for every measure. Here, it is beneficial to work with the smaller set of weak controls, which are subject to more structural constraints.
For upper hemicontinuity we use martingale problem and tightness arguments for laws of semimartingales. The proof for lower hemicontinuity relies on strong existence and stability properties of certain stochastic differential equations.

As explained above, in the Markovian case, regularity of the value function \(v\) allows us to identify the semigroup~\((S_t)_{t \in \bR_+}\) as a bounded viscosity solution to 
the nonlinear Kolmogorov-type partial differential equation
\begin{equation} \label{eq: intro PDE G}
\begin{cases}   
\partial_t u (t, x) - G (x, u(t, \cdot\,)) = 0, & \text{for } (t, x) \in \bR_+ \times \mathbb{R}^d, \\
u (0, x) = \psi (x), & \text{for } x \in \bR^d,
\end{cases}
\end{equation}
where the  nonlinearity \(G(x, \phi)\) is given by
\begin{align*}
    G(x, \phi) := \sup \Big\{ \langle \nabla & \phi (x), b (f, x) \rangle
    + \tfrac{1}{2} \on{tr} \big[ \nabla^2\phi (x) \sigma \sigma^* (f, x) \big]
    \\& + \int \big[ \phi(x + k (f, x, z)) - \phi(x) - \langle \nabla \phi(x) , h (k (f, x, z)) \rangle\big] \n (dz)
    \colon f \in F \Big\}
\end{align*}
for  \((x,\phi) \in \mathbb{R}^d \times C^{2}_b(\mathbb{R}^d; \bR)\).
In particular, for suitable input functions, the nonlinearity \(G\) provides a concise description of the (pointwise) generator of \((S_t)_{t \in \bR_+}\). 
Further, under global Lipschitz conditions, a comparison principle from \cite{hol16} identifies \(v\) as the unique bounded viscosity solution to the PDE \eqref{eq: intro PDE G}. This allows us to characterize \((S_t)_{t \in \bR_+}\) as the unique jointly continuous sublinear Markovian semigroup on \(C_b(\bR^d; \bR)\) whose (pointwise) generator coincides with \(G\) on the space of compactly supported smooth functions. 

\smallskip
This paper is structured as follows: in Section \ref{sec: setting} we introduce our setting and in Section~\ref{sec: regularity} we present our
regularity results for the value function.
In Section~\ref{sec: markovian} we discuss the Markovian case.
Section~\ref{sec: ex} is devoted to the study of random \(G\)-double exponential processes.
In Section~\ref{sec: control} we relate the value function to its counterparts from stochastic optimal control.
The regularity results for \(v\) are proved in Section~\ref{sec: pf regularity}, while the proofs for the Markovian case are given in Section \ref{sec: pf markovian}.

\newpage
\section{Nonlinear Semimartingales with Jumps}
In this section we present our main results on the regularity of the value function.

\subsection{The Framework}\label{sec: setting}
Let \(d \in \mathbb{N}\) be a fixed dimension and define $\Omega$ to be the space of all \cadlag functions \(\mathbb{R}_+ \to \mathbb{R}^d\) endowed with the Skorokhod \(J_1\) topology. The Euclidean scalar product and the corresponding Euclidean norm are denoted by \(\langle \cdot, \cdot\rangle\) and \(\|\cdot\|\).
We write \(X\) for the canonical process on $\Omega$, i.e., \(X_t (\omega) = \omega (t)\) for \(\omega \in \Omega\) and \(t \in \mathbb{R}_+\). 
It is well-known that \(\mathcal{F} := \mathcal{B}(\Omega) = \sigma (X_t, t \geq 0)\).
We define $\F := (\mathcal{F}_t)_{t \geq 0}$ as the canonical filtration generated by $X$, i.e., \(\mathcal{F}_t := \sigma (X_s, s \leq t)\) for \(t \in \mathbb{R}_+\). Further, let \(\mathscr{P}\) be the \(\F\)-predictable \(\sigma\)-field on \(\bR_+ \times \Omega\).
To lighten our notation, for two stopping times \(S\) and \(T\), we define the stochastic interval
\[
\of S, T \of \hspace{0.1cm} := \{ (t,\omega) \in \bR_+ \times \Omega \colon S(\omega) \leq t < T(\omega) \}.
\]
The stochastic intervals \( \gs S, T \of, \of S, T \gs, \gs S, T \gs  \) are defined accordingly.
In particular, the identity \(\of 0, \infty\of \hspace{0.1cm} = \bR_+ \times \Omega\) holds.

The set of probability measures on \((\Omega, \mathcal{F})\) is denoted by \(\mathfrak{P}(\Omega)\) and endowed with the usual topology of convergence in distribution.
Let \(\mathbb{S}^d_+\) be the set of all symmetric nonnegative semidefinite real-valued \(d \times d\) matrices, and set 
\[
\mathcal{L} := \Big\{ K \text{ measure on } (\bR^d, \mathcal{B}(\bR^d)) \colon K (\{0\}) = 0, \  \int \,(\|x\|^2 \wedge 1)\, K (dx) < \infty\Big\}.
\]
We endow \(\mathcal{L}\) with the weakest topology under that the maps
\[
K \mapsto \int f (x) (\|x\|^2 \wedge 1)\, K (dx), \quad f \in C_b (\bR^d; \bR), 
\]
are continuous. With this topology, the space \(\mathcal{L}\) is Polish, see \cite[Lemma~A.1]{C22a}.

Next, we introduce coefficients that capture differential characteristics under ambiguity. In \cite{hol16}, the author introduced uncertainty to the semimartingale characteristics via a state and control dependent L\'evy--Khinchine triplet \((b, a, K)\). In this paper, we use a push-forward density with respect to a given reference measure to introduce uncertainty to the third characteristic. A similar idea was used in the final section of \cite{nicole1987compactification}. 
As we explain now, this point of view is equivalent to those from \cite{hol16} in many aspects.
It is classical (see, e.g., \cite[Lemma~3.4, Remark~3.6]{CJ81}) that for any Borel kernel \(K\) from a Polish space \(E\) into \(\mathcal{L}\), there exists a measurable map \(k \colon E \times \bR \to \bR^d\) such that 
\begin{align}\label{eq: K rep mit usual ref me}
K (x, G) = \int \1_{G \backslash \{0\}} (k (x, y)) dy/y^2, \quad G \in \mathcal{B} (\bR^d).
\end{align}
From a modeling point of view, it is equivalent to model the kernel \(K\) or its push-forward density \(k\). The choice of the reference measure \(dy/ y^2\) is flexible and we will work with a general measure \(\n\) to allow for an additional degree of freedom.

\begin{remark} \label{rem: k}
We consider the case \(d = 1\). Let \(\nu\) be an infinite, non-atomic Borel measure on \(\bR\) such that \( (- \infty, 0) \ni t \mapsto \nu ((- \infty, t])\) and \((0, \infty) \ni t \mapsto \nu ([t, \infty))\) are strictly increasing and finite. Define 
    \begin{align*}
    k (x, y) := \begin{cases} 
\sup \big\{ z > 0 \colon K (x, [z, \infty)) \geq \nu ([y, \infty)) \big\}, & y > 0, 
\\
0, & y = 0, 
\\
\sup \big\{ z < 0 \colon K (x, (- \infty, z]) \geq \nu ((- \infty, y]) \big\}, & y < 0,
    \end{cases}
    \end{align*}
with the convention \(\sup \emptyset = 0\).    
Then, provided \(K (x, \{0\}) = 0\), \(K (x, [t, \infty)) < \infty\) and \(K (x, (- \infty, - t]) < \infty\) for all \(t > 0\), we have 
\[
K (x, G) = \int \1_{G \backslash \{0\}} (k (x, y)) \nu (dy), \quad G \in \mathcal{B} (\bR), 
\]
cf. \cite[Remark~3.6, p. 194]{CJ81} or \cite[Exercise~9.2.6]{stroockanalytic}.

In certain cases, it is not necessary that \(\nu\) is infinite and it suffices that \(\nu\) has more mass than every \(K (\,\cdot\,, dy)\). Such a case is discussed in Section~\ref{sec: ex} below.
\end{remark}

Take a topological space \(F\), a dimension \(r \in \mathbb{N}\) and a Lusin space \(L\).
Let \(\llambda \otimes \n\) be the intensity measure of a Poisson random measure on \(\bR_+ \times L\), where \(\llambda\) denotes the Lebesgue measure.
Further, fix two \(\mathcal{B}(F) \otimes \mathscr{P}\)-measurable functions \(b \colon F \times \of 0, \infty\of \hspace{0.05cm} \to \bR^d\) and \(\sigma \colon F \times \of 0, \infty\of \hspace{0.05cm} \to \bR^{d \times r}\) and a \(\mathcal{B}(F) \otimes \mathscr{P} \otimes \mathcal{B}(L)\)-measurable function \(k \colon F \times \of 0, \infty\of \hspace{0.05cm} \times\hspace{0.05cm} L \to \bR^d\) such that, for all \((f, t, \omega) \in F\times\, \of 0, \infty\of\), 
\begin{align} \label{eq: assumption k inte}
\int \,(\|k (f, t, \omega, z)\|^2 \wedge 1)\, \n (dz) < \infty.
\end{align}
For a Borel function \(g \colon L \to \mathbb{R}^d\), we define \(\n \o g^{-1}\) to be the Borel measure
\[
\n \o g^{-1} (G) := \int \1_{G \backslash \{0\}} (g (z))\, \n (dz), \quad G \in \mathcal{B}(\bR^d). 
\]
Further, we define a correspondence, i.e., a set-valued mapping, \(\Theta \colon \of 0, \infty\of \hspace{0.05cm} \twoheadrightarrow \mathbb{R}^d \times \bS \times \mathcal{L}\) by
\[
\Theta (t, \omega) := \big\{ \big(b (f, t, \omega), \sigma \sigma^* (f, t, \omega), \n \o k (f, t, \omega, \cdot)^{-1} \big) \colon f \in F \big\}.
\]
In this paper, we work under the following
\begin{SA} \label{SA: meas gr}
\(\Theta\) has a measurable graph, i.e., the graph
\[
\on{gr} \Theta = \big\{ (t, \omega, b, a, K) \in \of 0, \infty\of \hspace{0.05cm} \times \hspace{0.05cm} \mathbb{R}^d \times \bS \times \mathcal{L} \colon (b, a, K) \in \Theta (t, \omega) \big\}
\]
is Borel. 
\end{SA}
\begin{remark}
By \cite[Lemma 2.9]{CN22a}, Standing Assumption \ref{SA: meas gr} holds under Condition \ref{cond: main1} (i) and (iii) below. 
\end{remark}

We call an \(\bR^d\)-valued \cadlag process \(Y = (Y_t)_{t \geq 0}\) a \emph{semimartingale after a time \(t^* \in \mathbb{R}_+\)} if the process \(Y_{\cdot + t^*} = (Y_{t + t^*})_{t \geq 0}\) is a \(d\)-dimensional semimartingale for its natural right-continuous filtration.
The law of a semimartingale after \(t^*\) is said to be a \emph{semimartingale law after \(t^*\)} and the set of them is denoted by \(\fPs (t^*)\).
For \(P \in \fPs (t^*)\) we denote the semimartingale characteristics (with respect to a fixed {\em Lipschitz continuous} truncation function \(h \colon \bR^d \to \bR^d\)) of the shifted coordinate process \(X_{\cdot + t^*}\) by \((B^P_{\cdot + t^*}, C^P_{\cdot + t^*}, \nu^P_{\cdot + t^*})\). 
Moreover, we set 
\[
\fPas (t^*) := \big\{ P \in \fPs (t^*) \colon P\text{-a.s. } (B^P_{\cdot + t^*}, C^P_{\cdot + t^*}, \nu^P_{\cdot +t^*}) \ll \llambda \big\}, \quad \fPas := \fPas (0).
\]
Finally, for $(t,\omega) \in \of 0, \infty\of$, we define $\cA(t,\omega) \subset \mathfrak{P}(\Omega)$ by
\begin{align*}
\cC(t,\omega) := \Big\{ P \in \mathfrak{P}_{\text{sem}}^{\text{ac}}(t)\colon &P(X = \omega \text{ on } [0, t]) = 1, \\&\quad
(\llambda \otimes P)\text{-a.e. } (dB^P_{\cdot + t} /d\llambda, dC^P_{\cdot + t}/d\llambda, \nu^P_{\cdot + t}/ d \llambda) \in \Theta (\cdot + t, X) \Big\}.
\end{align*}

From now on, we also impose the following
\begin{SA} \label{SA: non empty}
\(\cA (t, \omega) \not = \emptyset\) for all \((t, \omega) \in \of 0, \infty\of\).
\end{SA}

\begin{remark} \label{rem: non empty}
    Standing Assumption \ref{SA: non empty} can be deduced from existence theorems for semimartingale laws as given in \cite{criens19,jacod79,JS}. For instance, \cite[Theorem IX.2.31]{JS} shows that Standing Assumption \ref{SA: non empty} holds under the boundedness conditions from Condition \ref{cond: main1} (iv) below and mild continuity conditions on \(b, \sigma\) and \(k\) as detailed in \cite[IX.2.27, IX.2.28]{JS}. For the Markovian case, where \(b, \sigma\) and \(k\) depend on \((t, \omega)\) only through \(\omega (t-)\), we also refer to \cite[Corollary~IX.2.33]{JS}.
\end{remark}

Suppose that \(\psi \colon \Omega \to [- \infty, \infty]\) is an upper semianalytic function, i.e., \(\{\omega \in \Omega \colon \psi (\omega) > c\}\) is analytic for every \(c \in \mathbb{R}\),
and define the so-called \emph{value function} by
\begin{align} \label{eq: def vf}
v (t, \omega) := \sup_{P \in \cA (t, \omega)} E^P \big[ \psi \big], \quad (t, \omega) \in \of 0, \infty\of.
\end{align}
Under the Standing Assumptions \ref{SA: meas gr} and \ref{SA: non empty}, the value function satisfies the dynamic programming principle as given by the following theorem. Its proof is similar to the one of \cite[Theorem~3.1]{CN22a} for the path-continuous case. Alternatively, the result can be deduced from a suitable version of \cite[Theorem~4.29]{hol16}. We omit the details.

\begin{theorem}[Dynamic Programming Principle] \label{theo: DPP}
The value function \(v\) is upper semianalytic. Moreover, for every pair \((t, \omega) \in \of 0, \infty\of \) and every stopping time \(\tau\) with \(t \leq \tau < \infty\), we have 
\begin{align} \label{eq: DPP}
v(t, \omega) = \sup_{P \in \cA (t, \omega)} E^P \big[ v (\tau, X) \big]. 
\end{align}
\end{theorem}

The DPP provides important properties of the value function. In many cases, for instance to characterize its dynamics, it is important to establish in addition certain regularity properties of the value function. In the second part of this section, we provide explicit conditions on the coefficients \(b, a\) and \(v\) for upper and lower semicontinuity of the value function.

\subsection{Regularity of the Value Function} \label{sec: regularity}
In this section we present conditions for upper and lower semicontinuity of the function \([0, T] \times D([0, T]; \bR^d) \ni (t, \omega) \mapsto v (t, \omega)\), where \(D([0, T]; \bR^d)\) denotes the space of all \cadlag functions from \([0, T]\) into \(\bR^d\) and \(T > 0\) is an arbitrary finite time horizon. We define a pseudometric \(\d\) on the space \([0, T] \times D([0, T]; \bR^d)\) by the formula
\begin{align}\label{eq: qseudo metric}
\d ((t, \omega), (s, \alpha)) := |t - s| + \sup_{r \in [0, T]} \| \omega (r \wedge t) - \alpha (r \wedge s)\|,
\end{align}
where \((t, \omega), (s, \alpha) \in [0, T] \times D([0, T]; \bR^d)\). Evidently, \(\d\) is not a metric, as \(\d ((t, \omega), (s, \alpha)) = 0\) does not imply \(\omega = \alpha\). Nevertheless, when passing to equivalence classes, i.e., identifying all elements \((t, \omega)\) and \((s, \alpha)\) such that \(\d ((t, \omega), (s, \alpha)) = 0\), the space \(([0, T] \times D([0, T]; \bR^d), \mathsf{d})\) becomes a metric space. From now on, we use this metric structure for \([0, T] \times D([0, T]; \bR^d)\).

\begin{condition} \label{cond: main1}
    \quad
    \begin{enumerate}
        \item[\textup{(i)}] \(F\) is a compact metrizable space.
        \item[\textup{(ii)}] \(\Theta\) is convex-valued.
        \item[\textup{(iii)}] For every \((t, \omega) \in \of 0, \infty\of\), the map \(f \mapsto (b (f, t, \omega), \sigma \sigma^* (f, t, \omega), \n \o k (f, t, \omega, \cdot)^{-1})\) is continuous from \(F\) into \(\bR^d \times \bS \times \mathcal{L}\).
        \item[\textup{(iv)}] For every \(T > 0\), there exists a constant \(C = C_T > 0\) such that 
        \[
        \|b (f, t, \omega)\| + \|\sigma (f, t, \omega)\| + \int \,(\|k (f, t, \omega, z)\|^2 \wedge 1)\, \n (dz) \leq C
        \]
        for all \((f, t, \omega) \in F \times \of 0, T\gs\). Moreover, for all \(T > 0\),
        \begin{align} \label{eq: tightness in cond}
        \sup \big\{ \n (\|k (f, t, \omega, \cdot ) \| > R ) \colon (f, t, \omega) \in F \times \of 0, T\gs \big\} \to 0 \text{ as } R \to \infty.
        \end{align}
    \end{enumerate}
\end{condition}

  Let \(C^2_b (\bR^d; \bR)\) be the set of all bounded twice continuously differentiable functions from \(\bR^d\) into \(\bR\) with bounded gradient and Hessian matrix. For \(g \in C^2_b (\bR^d; \bR)\), we set 
  \begin{equation} \label{eq: def Lg}
  \begin{split}
\mathcal{L} g (f, t, \omega, x) &:= \langle \nabla  g (x), b (f, t, \omega) \rangle + \tfrac{1}{2} \on{tr} \big[ \nabla^2 g (x) \sigma \sigma^* (f, t, \omega)\big]
\\&\qquad \quad + \int \big[ g (x + k (f, t, \omega, z)) - g (x) - \langle \nabla g (x), h (k (f, t, \omega, z)) \rangle\big] \,\n (dz)
  \end{split}
  \end{equation}
  for \((f, t, \omega, x) \in F \times \of 0, \infty\of \, \times \hspace{0.05cm}\bR^d\). 
  Notice that the integral in \eqref{eq: def Lg} is well-defined due to our assumption that \eqref{eq: assumption k inte} holds.

\begin{condition} \label{cond: main2}
Condition \ref{cond: main1} holds and, for every \(t \in \bR_+\), \(g \in C^2_b (\bR^d; \bR)\), and every sequence \((\omega^n)_{n=0}^\infty \subset \Omega\) with
\(\omega^n \to \omega^0\) in the Skorokhod \(J_1\) topology, it holds that
\begin{align} \label{eq: conv cond 2.8}
\int_0^t  \sup_{f \in F} | \mathcal{L} g (f, s, \omega^n, \omega^n (s)) - \mathcal{L} g (f, s, \omega^0, \omega^0 (s)) | ds \to 0.
\end{align}
\end{condition}

To state our first main result, we require one more definition.
For a set \(\mathcal{Q}\) of probability measures, we say that a property holds {\em \(\mathcal{Q}\)-a.s.} if it holds \(Q\)-a.s. for all \(Q \in \mathcal{Q}\). Note that, compared to the more frequently used notion of {\em \(\mathcal{Q}\)-quasi surely}, the exception set is allowed to depend on \(Q \in \mathcal{Q}\).
We are in the position to formulate the first main result of this section.

\begin{theorem} \label{theo: value function upper semi}
Suppose that Condition \ref{cond: main2} holds. Then, for every \(T > 0\), the value function \(v \colon [0, T] \times D([0, T]; \bR^d) \to \bR\) is upper semicontinuous at \((t, \omega) \in [0, T] \times D ([0, T]; \bR^d)\) for every bounded \(\mathcal{C}(t, \omega)\)-a.s. upper semicontinuous input function \(\psi \colon \Omega \to \bR\).
\end{theorem}
Theorem \ref{theo: value function upper semi} extends \cite[Theorems~4.3 and 4.4]{CN22a} on the upper semicontinuity of the value function from a path-continuous framework to a setting with jumps. Additionally, compared to the results of \cite{CN22a}, Theorem~\ref{theo: value function upper semi} allows for input functions \(\psi\) that are only \(\cC\)-a.s. upper semicontinuous. In the presence of jumps this is an important generalization, as functions of the type \(\omega \mapsto \psi (\omega (t))\) are often \(\cC\)-a.s. but not everywhere upper semicontinuous, cf. Remark~\ref{rem: as continuity} below.

\smallskip
Next, we also investigate lower semicontinuity of the value function. 

\begin{condition} \label{cond: main3}
For every \(N > 0\), there is a constant \(C = C_N > 0\) and a Borel function \(\gamma = \gamma_N \colon L \to [0, \infty]\) such that \(\int (\gamma (z) \vee \gamma^2 (z))\, \n (dz) < \infty\) and 
\begin{align*}
\|b (f, t, \omega) - b (f, t, \alpha)\| + \|\sigma (f, t, \omega) - \sigma (f, t, \alpha) \| &\leq C \sup_{s \in [0, t]} \|\omega (s) - \alpha (s)\|,  \\
\|k (f, t, \omega, z) - k (f, t, \alpha, z)\| &\leq \gamma (z)  \sup_{s \in [0, t]} \|\omega (s) - \alpha (s)\|
\end{align*}
for all \((f, t, z) \in F \times [0, N] \times L\) and \(\omega, \alpha \in \Omega \colon \sup_{s \in [0, t]} \|\omega (s)\| \vee \|\alpha (s)\| \leq N\). Furthermore, there exists a Borel function \(\beta \colon L \to [0, \infty]\) such that \(\int (\beta (z) \vee \beta^2 (z)) \, \n (dz) < \infty\) and 
\[
\| k (f, t, \omega, z)\| \leq \beta (z) 
\]
for all \((f, t, \omega, z) \in F \times \of 0, \infty\of \hspace{0.05cm}\times \hspace{0.05cm} L\).
\end{condition}

\begin{remark}
The Conditions \ref{cond: main1} and \ref{cond: main2} are independent of our choice to model the third differential characteristic via a reference measure \(\n\) and a density coefficients \(k\). Indeed, the conditions would be exactly the same if we start with a predictable transition kernel \(K\) from \(\of 0, \infty\of\) into \((\bR^d, \mathcal{B}(\bR^d))\). In contrast, the coefficient \(k\) is explicitly used in Condition~\ref{cond: main3}, which is the main motivation for our modeling choice.
\end{remark}

\begin{theorem} \label{theo: value function lower semi}
    Suppose that the Conditions \ref{cond: main1} and \ref{cond: main3} hold. Then, for every \(T > 0\), the value function \(v \colon [0, T] \times D([0, T]; \bR^d) \to \bR\) is lower semicontinuous at \((t, \omega) \in [0, T] \times D ([0, T]; \bR^d)\) for every bounded \(\mathcal{C}(t, \omega)\)-a.s. lower semicontinuous input function \(\psi \colon \Omega \to \bR\).
\end{theorem}

Theorem \ref{theo: value function lower semi} extends  \cite[Theorems~4.3 and 4.7]{CN22a} on the lower semicontinuity of the value function from a path-continuous to a setting with jumps. Moreover, it gives a positive answer to a problem posed in \cite[Remark~4.43]{hol16}, as it establishes lower semicontinuity of the value function. We refer to \cite[Remark~3.4]{K19} and 
\cite[Remark~5.4]{K21} for further comments.

\smallskip
Putting the Theorems~\ref{theo: value function upper semi} and \ref{theo: value function lower semi} together, we obtain the following result on the continuity of the value function.

\begin{corollary}
    Suppose that the Conditions \ref{cond: main2} and \ref{cond: main3} hold. Then, for every \(T > 0\), the value function \(v \colon [0, T] \times D([0, T]; \bR^d) \to \bR\) is continuous at \((t, \omega) \in [0, T] \times D ([0, T]; \bR^d)\) for every bounded \(\mathcal{C}(t, \omega)\)-a.s. continuous input function \(\psi \colon \Omega \to \bR\).
\end{corollary}

\begin{remark} \label{rem: as continuity}
Notice that \(X_t\) is neither upper nor lower semicontinuous.
However, for \((s, \omega) \in \of 0, \infty\of\) and \(t \geq s\), \(X_t\) is \(\cC (s, \omega)\)-a.s. continuous by~\cite[Corollary II.1.19; VI.2.3]{JS}.
In particular, for an upper (lower) semicontinuous function \(\psi \colon \mathbb{R}^d \to \bR\), the map  \(\psi (X_t)\) is \(\cC (s, \omega)\)-a.s. upper (lower) semicontinuous. 
\end{remark}

\section{The Markovian Case} \label{sec: markovian}
Throughout this section, we assume that the correspondence \( \Theta \) has a \emph{Markovian structure}, i.e., for every \( (t, \omega) \in \of 0, \infty \of \), the set \( \Theta(t, \omega) \) depends on \((t, \omega)\) only through the value~\( \omega(t-) \). More precisely, we assume the existence of three Borel functions \(\bb \colon F \times \bR^d \to \bR^d\), \(\bs \colon F \times \bR^d \to \bR^{d \times r}\) and \(\bk \colon F \times \bR^d \hspace{0.05cm} \times\hspace{0.05cm} L \to \bR^d\)
such that
\[
b(f,t,\omega) = \bb(f, \omega(t-)), \quad \sigma(f,t,\omega) = \bs(f, \omega(t-)), \quad 
k(f,t, \omega, z) = \bk(f, \omega(t-), z), 
\]
for all \((f,t, \omega) \in F \times \of 0, \infty \of\) and \(z \in L\).

As in the general path-dependent setting considered in Section \ref{sec: regularity}, we introduce appropriate assumptions on the coefficients \(\bb, \bs\) and \(\bk\) in order to obtain regularity results.

\begin{condition} \label{cond: main1 markov}
    \quad
    \begin{enumerate}
        \item[\textup{(i)}] \(F\) is a compact metrizable space.
        \item[\textup{(ii)}] \(\Theta\) is convex-valued.
        \item[\textup{(iii)}] The map \((f, x) \mapsto (\bb (f, x), \bs \bs^* (f, x), \n \o \bk (f, x, \cdot)^{-1})\) is continuous from \(F \times \bR^d\) into \(\bR^d \times \bS \times \mathcal{L}\).
        \item[\textup{(iv)}] There exists a constant \(C > 0\) such that 
        \[
        \|\bb (f, x)\| + \|\bs (f,x)\| + \int \,(\|\bk (f, x, z)\|^2 \wedge 1)\, \n (dz) \leq C
        \]
        for all \((f, x) \in F \times \bR^d\). Moreover,
        \[
        \sup \big\{ \n (\|\bk (f, x, \cdot ) \| > R ) \colon (f, x) \in F \times \bR^d \big\} \to 0 \text{ as } R \to \infty.
        \]
    \end{enumerate}
\end{condition}

\begin{condition} \label{cond: main3 markov}
For every \(N > 0\), there is a constant \(C = C_N > 0\) and a Borel function \(\gamma = \gamma_N \colon L \to [0, \infty]\) such that \(\int (\gamma(z) \vee \gamma^2 (z)) \, \n (dz) < \infty\) and 
\begin{align*}
\|\bb (f, x) - \bb (f, y)\| + \|\bs (f, x) - \bs (f, y) \| &\leq C  \|x - y\|,  \\
\| \bk (f, x, z) - \bk (f, y, z)\| &\leq \gamma (z)  \|x - y\|
\end{align*}
for all \((f, z) \in F \times L\) and \(x,y \in \bR^d \colon \|x\| \vee \|y\| \leq N\). Furthermore, there exists a Borel function \(\beta \colon L \to [0, \infty]\) such that \(\int (\beta (z) \vee \beta^2 (z)) \, \n (dz) < \infty\) and 
\[
\| \bk (f, x, z)\| \leq \beta (z) 
\]
for all \((f, x, z) \in F \times \bR^d \hspace{0.05cm}\times \hspace{0.05cm} L\).
\end{condition}

\begin{remark} \label{rem: cond in markov}
Notice that the Conditions \ref{cond: main1} and \ref{cond: main2} hold true once Condition \ref{cond: main1 markov} is in force. Indeed, by parts (i) and (iii) of Condition \ref{cond: main1 markov} and Berge's maximum theorem (\cite[Theorem~17.31]{charalambos2013infinite}), for every convergent sequence \(\omega^n \to \omega^0\), the difference
\[
\sup_{f \in F}| \mathcal{L} g (f, s, \omega^n, \omega^n (s)) - \mathcal{L} g (f, s, \omega^0, \omega^0 (s)) |
\]
vanishes for \(\llambda\)-a.e. \(s \in \bR_+\) as \(n \to \infty\). Hence, thanks to the dominated convergence theorem, Condition \ref{cond: main2} holds in this case.
Further, Condition \ref{cond: main3 markov} guarantees Condition~\ref{cond: main3}.
\end{remark}

\subsection{The Semigroup and its Feller Properties}
For each \( x \in \bR^d \), let \(\bx \in \Omega\) be the constant path \(\bx (s) = x\) for all \(s \in \bR_+\).
We define the sublinear operator \( \cE^x \) on the convex cone of upper semianalytic functions \(\psi \colon \Omega \to [- \infty, \infty]\) by
\( \cE^x(\psi) := \sup_{P \in \cC(0,\bx)} E^P \big[ \psi \big] \). For every \( x \in \bR^d \), we have by construction that \( \cE^x(\psi(X_0)) = \psi(x) \) for every upper semianalytic function \(\psi \colon \bR^d \to [- \infty, \infty]\).

\begin{definition}
Let \( \mathcal{H} \) be a convex cone of functions \( f \colon \bR^d \to \bR \) containing all constant functions.
A family of sublinear operators \( S_t \colon \mathcal{H} \to \mathcal{H}, \ t \in \bR_+,\) is called a \emph{sublinear Markovian semigroup} on \( \mathcal{H} \) if it satisfies the following properties:
\begin{enumerate}
    \item[\textup{(i)}] \( (S_t)_{t \in \bR_+} \) has the semigroup property, i.e., 
          \( S_s S_t = S_{s+t} \) for all \(s, t \in \bR_+ \)  and
          \( S_0 = \on{id} \),
    
    \item[\textup{(ii)}] \( S_t \) is monotone for each \( t \in \bR_+\), i.e., 
    \( f, g \in \mathcal{H} \) with \( f \leq g \) implies \(S_t(f) \leq S_t(g) \),
    
    \item[\textup{(iii)}] \( S_t \) preserves constants for each  \( t \in \bR_+\), i.e.,
    \( S_t(c) = c \) for each \( c \in \bR  \).
\end{enumerate}
\end{definition}

Denote, for \(t \in \mathbb{R}_+\), the shift operator \(\theta_t \colon \Omega \to \Omega\) by \(\theta_t (\omega) := \omega(\cdot + t)\) for all \(\omega \in \Omega\).
The next result provides the nonlinear Markov property of the family
\( \{ \cE^x \colon x \in \bR^d\}\), cf. \cite[Lemma 4.32]{hol16} and \cite[Proposition 2.8]{CN22b}. We omit a detailed proof.

\begin{proposition} \label{prop: markov property}
For every upper semianalytic function \( \psi \colon \Omega \to [- \infty, \infty] \), the equality
\[
\cE^x( \psi \circ \theta_t) = \cE^x ( \cE^{X_t} (\psi))
\]
holds for every \((t, x) \in \bR_+ \times \bR^d\).
\end{proposition}

Proposition \ref{prop: markov property} confirms the intuition that the coordinate process is a nonlinear Markov process under the family \( \{\cE^x \colon x \in \bR^d\} \), as it implies the equality
\[ \cE^x(\psi(X_{s+t})) = \cE^x( \cE^{X_t}( \psi (X_s))) \]
for every upper semianalytic function \(\psi \colon \bR^d \to [- \infty, \infty]\), \(s,t \in \bR_+ \) and \(x \in \bR^d\).

We introduce a family of operators \( (S_t)_{t \in \bR_+} \) given by
\begin{equation} \label{eq: semigroup}
    S_t ( \psi )(x) := \cE^x(\psi(X_t)), \quad (t, x) \in \bR_+ \times \bR^d.
\end{equation}

The following result is an immediate consequence of Proposition~\ref{prop: markov property}. It is a restatement of \cite[Remark~4.33]{hol16} for our framework.

\begin{proposition} 
The family of operators \( (S_t)_{t \in \bR_+} \) from \eqref{eq: semigroup}
defines a sublinear Markovian semigroup on the set of bounded upper semianalytic functions.
\end{proposition}

In the following, we investigate the semigroup property of \((S_t)_{t \in \mathbb{R}_+}\) on convex cones consisting of more regular functions.

\begin{definition}
We say that the sublinear Markovian semigroup \((S_t)_{t \in \mathbb{R}_+}\) has the 
\begin{enumerate}
    \item[\textup{(a)}]
    \emph{\(\usc_b\)--Feller property} if it is a sublinear Markovian semigroup on the space \(\usc_b(\mathbb{R}^d; \mathbb{R})\) of bounded upper semicontinuous functions from \(\mathbb{R}^d\) into \(\mathbb{R}\);
    \item[\textup{(b)}]
       \emph{\(\lsc_b\)--Feller property} if it is a sublinear Markovian semigroup on the space \(\lsc_b(\mathbb{R}^d; \mathbb{R})\) of bounded lower semicontinuous functions from \(\mathbb{R}^d\) into \(\mathbb{R}\);
    \item[\textup{(c)}]
     \emph{\(C_b\)--Feller property} if it is a sublinear Markovian semigroup on the space \(C_b(\mathbb{R}^d; \mathbb{R})\) of bounded continuous functions from \(\mathbb{R}^d\) into \(\mathbb{R}\).

\end{enumerate}
\end{definition}

As observed in \cite{schilling98}, in case of linear semigroups, the \(\usc_b\)–Feller property is
equivalent to the \(C_b\)–Feller property. Indeed, this follows simply from the fact that 
\(\usc_b(\bR^d; \bR) = - \lsc_b(\bR^d; \bR)\).

Together with Remark \ref{rem: as continuity}, the regularity of the value function \(v\) established in Theorems~\ref{theo: value function upper semi} and \ref{theo: value function lower semi} grants access to Feller properties of the associated semigroup \( (S_t)_{t \in \bR_+}\). The \(C_b\)--Feller property is of fundamental importance for the relation of
nonlinear processes and semigroups. In the next section, it enables us to link the semigroup \( (S_t)_{t \in \bR_+}\) to a certain Kolmogorov type PDE.

For a nonlinear framework with jumps, the \(\usc_b\)--Feller property property was given in  \cite[Theorem~4.41, Lemma~4.42]{hol16}.
For nonlinear one-dimensional diffusions, a version of the following theorem was established in \cite{CN22b}. 
Theorem~\ref{thm: feller} extends results for the L\'evy case from \cite{hol16,K19,neufeld2017nonlinear} to a Markovian framework.

\begin{theorem} \label{thm: feller}
Consider the family of operators  \( (S_t)_{t \in \bR_+} \) as given in \eqref{eq: semigroup}.
Then, the following statements hold:
\begin{enumerate}
    \item[\textup{(i)}]  \((S_t)_{t \in \bR_+}\) has the \(\usc_b\)--Feller property, if the Condition \ref{cond: main1 markov} holds;

    \item[\textup{(ii)}] \((S_t)_{t \in \bR_+}\) has the \(C_b\)--Feller property, if the Conditions \ref{cond: main1 markov} and \ref{cond: main3 markov} hold. 
\end{enumerate}
\end{theorem}

The proof of Theorem \ref{thm: feller} is given in Section \ref{sec: pf feller}.

\subsection{Generator Equation}
For a sublinear Markovian semigroup \( (T_t)_{t \in \bR_+} \) on a convex cone~\(\cH\), its \emph{pointwise (infinitesimal) generator} \( A \colon  \cD(A) \to \cH \)  is defined by 
\begin{align*}
A(\phi)(x) & := \lim_{t \to 0} \frac{T_t(\phi)(x) - \phi(x)}{t}, \quad x \in \bR^d, \ \phi \in \cD(A), \\
\cD(A) &:= \Big\{ \phi \in \cH \colon \exists g \in \cH \text{ such that }  \lim_{t \to 0} \frac{T_t(\phi)(x) - \phi(x)}{t} = g(x) \ \ \forall x \in \bR^d \Big\}.
\end{align*}

As in the linear case, the semigroup of a nonlinear Markov process can be linked to its pointwise generator through an evolution equation.
Let us start with a formal introduction to the class of nonlinear PDEs under consideration.
Given a sublinear Markovian semigroup \( (T_t)_{t \in \bR_+} \) with pointwise generator \(A\), consider the equation
\begin{equation} \label{eq: PDE A}
\begin{cases}   
\partial_t u (t, x) - A(u(t, \cdot\, ))(x) = 0, & \text{for } (t, x) \in \bR_+ \times \mathbb{R}^d, \\
u (0, x) = \psi (x), & \text{for } x \in \bR^d,
\end{cases}
\end{equation}
where \(\psi \colon \bR^d \to \mathbb{R}\) is a suitable function.

Recall that a function \(u \colon \bR_+ \times \bR^d \to \mathbb{R}\) is said to be a \emph{viscosity subsolution} to \eqref{eq: PDE A} if it is upper semicontinuous and the following two properties hold:
\begin{enumerate}
\item[\textup{(a)}] \(u(0, \cdot\,) \leq \psi\);
\item[\textup{(b)}]
\(
\partial_t \phi (t, x) - A (u(t, \cdot\,))(x) \leq 0
\)
for all \(\phi \in  C^\infty_b(\bR_+ \times \bR^d; \bR)\) such that \(u \leq \phi\) and \(\phi (t, x) = u(t, x)\) for some \((t, x) \in (0, \infty) \times \bR^d \). 
\end{enumerate}
A \emph{viscosity supersolution} is obtained by replacing upper with lower semicontinuity and reversing the inequalities. Further, \(u\) is called \emph{viscosity solution} if it is both, a viscosity sub- and supersolution.

\smallskip
The following result is a restatement of \cite[Proposition 4.10]{hol16}.

\begin{proposition} \label{prop: weak sense}
Let \((T_t)_{t \in \bR_+}\) be a sublinear Markovian semigroup on a convex cone \(\cH\) of real-valued functions on \(\bR^d\)
containing all constant functions, and let \(A \colon \cD(A) \to \cH\) be its pointwise generator. 
Consider \(\psi \in \cH\) and assume that \((t,x)\mapsto T_t(\psi)(x)\) is continuous.
If \(C^\infty_b(\bR^d; \bR) \subset \cD(A)\),
then 
\[
\bR_+ \times \bR^d \ni (t, x) \mapsto T_t(\psi)(x)
\]
is a viscosity solution to \eqref{eq: PDE A}.
\end{proposition}
Next, we derive a formal description of the pointwise generator associated to the sublinear Markovian semigroup \((S_t)_{t \in \bR_+}\) defined in \eqref{eq: semigroup}. To this end, we define the nonlinearity \(G(x, \phi)\)
for  \((x,\phi) \in \mathbb{R}^d \times C^{2}_b(\mathbb{R}^d; \bR)\)
by
\begin{align*}
    G(x, \phi) :=  \sup \Big\{ \langle \nabla &\phi (x), \bb (f, x) \rangle
    + \tfrac{1}{2} \on{tr} \big[ \nabla^2\phi (x) \bs \bs^* (f, x) \big]
    \\& + \int\, \big[ \phi(x + \bk (f, x, z)) - \phi(x) - \langle \nabla \phi(x) , h (\bk (f, x, z)) \rangle\big]\, \n (dz)
    \colon f \in F \Big\}
\end{align*}
and consider the nonlinear Kolmogorov type partial differential equation
\begin{equation} \label{eq: PDE G}
\begin{cases}   
\partial_t u (t, x) - G (x, u(t, \cdot\,)) = 0, & \text{for } (t, x) \in \bR_+ \times \mathbb{R}^d, \\
u (0, x) = \psi (x), & \text{for } x \in \bR^d,
\end{cases}
\end{equation}
where, again, \(\psi \colon \bR^d \to \bR\) is a suitable function.

\begin{proposition} \label{prop: generator}
Suppose that Condition \ref{cond: main1 markov} holds.
Let  \((S_t)_{t \in \bR_+}\) be the family of operators
defined in \eqref{eq: semigroup}. 
Then, for every \(\phi \in C_b^2(\bR^d;\bR)\) such that \(\nabla^2\phi \) is uniformly continuous,
we have
\[
\lim_{t \to 0} \frac{S_t(\phi)(x) - \phi(x)}{t} = G(x, \phi), \quad x \in \bR^d.
\]

\end{proposition}

We defer the proof of Proposition \ref{prop: generator} to Section \ref{sec: pf generator}. As a consequence, we are able to identify the semigroup \((S_t)_{t \in \bR_+}\) as a viscosity solution to \eqref{eq: PDE G} once we have established joint continuity in time and space. This is the content of the next theorem. Its proof can be found in Section \ref{sec: pf viscosity}.

\begin{theorem} \label{thm: viscosity}
Suppose that the Conditions \ref{cond: main1 markov} and \ref{cond: main3 markov} hold, and let  \((S_t)_{t \in \bR_+}\) be the family of operators
defined in \eqref{eq: semigroup}.
Then, for every \(\psi \in C_b(\bR^d; \bR)\),
\[
\bR_+ \times \bR^d \ni (t, x) \mapsto S_t(\psi)(x)
\]
is a viscosity solution to \eqref{eq: PDE G}.
\end{theorem}

In the following we apply a uniqueness theorem for Hamilton–Jacobi–Bellman PDEs and show that, under global Lipschitz and boundedness
conditions, the semigroup \((S_t)_{t \in \bR_+}\) is the unique viscosity solution to \eqref{eq: PDE G}.

\begin{condition} \label{cond: global lipschitz}
There is a constant \(C > 0\) and a Borel function \(\gamma \colon L \to [0, \infty]\) such that \(\int \gamma (z) \vee \gamma^2(z) \, \n (dz) < \infty\) and 
\begin{align*}
\|\bb(f,x)- \bb(f,y) \| + \| \bs(f,x)-\bs(f,y) \| &\leq C  \|x -y \|,  \\
\| \bk (f, x, z) - \bk (f, y, z)\| &\leq  \gamma(z)  \|x - y\|, \\
\| \bk (f, x, z)\| & \leq  \gamma(z)
\end{align*}
for all \(f \in F\), \(x,y \in \bR^d\) and \(z \in L\).
\end{condition}

\begin{theorem} \label{thm: viscosity uniqueness}
Suppose that the Conditions \ref{cond: main1 markov} and \ref{cond: global lipschitz} hold, and let  \((S_t)_{t \in \bR_+}\) be as in~\eqref{eq: semigroup}.
Then, for every \(\psi \in C_b(\bR^d; \bR)\),
\[
\bR_+ \times \bR^d \ni (t, x) \mapsto S_t(\psi)(x)
\]
is the unique bounded viscosity solution to \eqref{eq: PDE G}.
\end{theorem}
The proof of Theorem \ref{thm: viscosity uniqueness} is relegated to Section \ref{sec: pf viscosity}. 
We end this section with a second uniqueness result. More precisely, we show that \((S_t)_{t \in \bR_+}\) is determined by its generator on \(C_c^\infty(\bR^d; \bR)\), the space of compactly supported smooth functions, in the following sense:

\begin{corollary} \label{cor: characterization}
Suppose that the Conditions \ref{cond: main1 markov} and \ref{cond: global lipschitz} hold. Then,
\((S_t)_{t \in \bR_+}\) as defined in~\eqref{eq: semigroup} is the unique sublinear Markovian semigroup on \(C_b(\bR^d; \bR)\) with the following properties:
\begin{enumerate}
    \item[\textup{(i)}] for \(\psi \in C_b(\bR^d; \bR)\), the map \((t,x) \mapsto S_t(\psi)(x)\) is continuous;
    \item[\textup{(ii)}] for all \(x \in \bR^d\) and \(\phi \in C_c^\infty(\bR^d; \bR)\),
    \[
    \lim_{t \to 0} \frac{S_t (\phi)(x) - \phi (x)}{t} = G(x, \phi).
    \]
\end{enumerate}
\end{corollary}
\begin{proof}
Theorem \ref{thm: feller} shows that \((S_t)_{t \in \bR_+}\) has the \(C_b\)--Feller property. Further, 
Theorem \ref{thm: viscosity} implies (i), while Proposition \ref{prop: generator} shows that (ii) holds for \((S_t)_{t \in \bR_+}\).
Next, let \((T_t)_{t \in \bR_+}\) be a sublinear Markovian semigroup on \(C_b(\bR^d; \bR)\) with the properties described in (i) and (ii).
Take an arbitrary  \(\psi \in C_b(\bR^d; \bR)\).
Using Lemma \ref{lem: viscosity} below, it follows that \((t,x) \mapsto T_t(\psi)(x)\) is a (bounded) viscosity solution to \eqref{eq: PDE G}.
Hence, Theorem~\ref{thm: viscosity uniqueness} implies
\[
S_t(\psi)(x) = T_t(\psi)(x), \quad (t,x) \in \bR_+ \times \bR^d,
\]
as desired.
\end{proof}

To the best of our knowledge, the first theorem that uses viscosity methods to characterize sublinear semigroups via versions of the pointwise generator appeared in \cite{lions_nisio}. A general characterization of sublinear convolution semigroups (that correspond to nonlinear L\'evy processes) was proved in \cite{K21}. Corollary~\ref{cor: characterization} extends this result to more general semigroups build from path-dependent uncertainty sets.

	An analytic approach not relying on viscosity methods was recently proposed in \cite{blessing22}. 
	More specifically, the paper provides a unique characterization of sublinear semigroups (with certain properties) in terms of so-called \(\Gamma\)-generators.	It is an interesting open problem to investigate how this theory relates to Corollary~\ref{cor: characterization}.
	We leave this question open for future investigations.

\section{Example: Random \(G\)-Double Exponential Process} \label{sec: ex}
In this section, we discuss an extension of the random \(G\)-Brownian motion setting from~\cite{NVH} that includes jumps.
Consider a one-dimensional L\'evy process with L\'evy--Khinchine triple \((b, a, F)\), where 
\begin{align*}
F_\lambda (dx) = \lambda e^{- |x|} dx, \quad \lambda > 0.
\end{align*}
This is a special case of a so-called double exponential L\'evy process, which is related to Kou's model \cite{Kou} in finance. In the following, we introduce path-dependent uncertainty to its parameters \(b, a\) and \(\lambda\), where the model uncertainty is governed by random intervals.

First of all, let us fix \(L := \bR\). We model the interval boundaries via predictable functions
		\[
		\underline{b}, \overline{b} \colon \of 0, \infty\of \, \to \bR, \quad 
		\underline{a}, \overline{a} \colon \of 0, \infty\of \, \to \bR_+, \quad
  \underline{\lambda}, \overline{\lambda} \colon \of 0, \infty\of\,\to \bR_+, 
		\]
		such that
		\[
		\underline{b}_t(\omega) \leq \overline{b}_t(\omega), \quad 
		\underline{a}_t(\omega) \leq \overline{a}_t(\omega), 
  \quad
 \underline{\lambda}_t (\omega) \leq \overline{\lambda}_t (\omega),
  \quad 
		(t, \omega) \in \of 0, \infty \of.
		\]
Our goal is to describe the infinitesimal behaviour of the characteristics by the set-valued map 
\begin{align} \label{eq: theta example}
    \Theta (t, \omega) := [ \underline{b}_t (\omega), \overline{b}_t (\omega)] \times [ \underline{a}_t (\omega), \overline{a}_t (\omega)] \times \{ F_\lambda \colon \lambda \in [\underline{\lambda}_t (\omega), \overline{\lambda}_t(\omega)] \}.
\end{align}
Notice already that \(\Theta\) is convex-valued. 
Next, we explain how to incorporate this setting into our framework. 
We presume that there are constants \(0 < \lambda_* \leq \lambda^* < \infty\) such that 
\[
\lambda_* \leq \underline{\lambda}_t (\omega) \leq \overline{\lambda}_t (\omega) \leq \lambda^*, \quad \forall \, (t, \omega) \in \of 0, \infty\of.
\]
Then, we take \(F := [0, 1] \times [0, 1] \times [0,1]\) and set, for \(((f_1, f_2, f_3),t,\omega, z) \in F \times \of 0, \infty \of  \, \times\, \bR\),
\begin{align*}
		b ((f_1, f_2, f_3), t, \omega) &:= \underline{b}_t (\omega) + f_1 \cdot (\overline{b}_t (\omega) - \underline{b}_t (\omega)), \\
		a ((f_1, f_2, f_3), t, \omega) &:= \underline{a}_t (\omega) + f_2 \cdot (\overline{a}_t (\omega) - \underline{a}_t (\omega)),
        \\ k ((f_1, f_2, f_3), t, \omega, z) &:=
        \begin{cases} - \log (\lambda^*) + \log (\lambda ( (f_1, f_2, f_3), t, \omega)) + z, & z > 0, \\ 0, &  z = 0, \\
       \log (\lambda^*) - \log ( \lambda ( (f_1, f_2, f_3), t, \omega) ) + z, & z < 0,
        \end{cases}
\end{align*}
where
\begin{align*}
    \lambda ( (f_1, f_2, f_3) , t, \omega) := \underline{\lambda}_t (\omega) + f_3 \cdot (\overline{\lambda}_t (\omega) - \underline{\lambda}_t(\omega)).
\end{align*}
Further, we define 
\[
\n (dz) := F_{\lambda^*} (dz) = \lambda^* e^{- |z|} dz.
\]
By a short computation (based on a change of variable), we obtain that 
\begin{align} \label{eq: push example}
\int \1_{G \backslash \{0\}} ( k (f, t, \omega, z) ) \n (dz) = \int_{G} \lambda (f, t, \omega) e^{- |x|} dx, \quad G \in \mathcal{B}(\bR).
\end{align}
This shows that
\[
\Theta (t, \omega) = \big\{ (b (f, t, \omega), a (f, t, \omega), \n \o k (f, t, \omega, \cdot\,)^{-1}) \colon f \in F \big\}.
\]
In other words, we reformulated \eqref{eq: theta example} within the setting from the previous section.

\smallskip
In the following, we discuss the Conditions~\ref{cond: main1}, \ref{cond: main2} and \ref{cond: main3} for this example.
We start with Condition \ref{cond: main1}.
It is clear that \(F\) is a compact metrizable space. Further, by virtue of \eqref{eq: theta example}, \(\Theta\) is convex-valued. 
It follows from \eqref{eq: push example} and the structure of \(b, a\) and \(\lambda\), that Condition~\ref{cond: main1}~(iii) holds. 
In case \(\underline{b}, \overline{b}\) and \(\overline{a}\) are bounded (which we assume in the following), we also get that Condition~\ref{cond: main1}~(iv) holds. Let us elaborate on the jump part in more detail.
Using \eqref{eq: push example} and \(\lambda \leq \lambda^*\), we obtain that
\begin{align*}
    \int (|k (f, t, \omega, z)|^2 \wedge 1)\, \n (dz) &= \int (|x|^2 \wedge 1) \lambda (f, t, \omega) e^{- |x|} dx 
    \leq \lambda^* \int e^{- |x|} dx < \infty. 
\end{align*}
Further, notice that
\begin{align*}
    \n ( \{ x \colon | k (f, t, \omega, x) | > R \} ) &= \int \1_{\{|x| > R\}} \lambda (f, t, \omega) e^{- |x|} dx \leq \n (\{x \colon |x| > R\}) \to 0,
\end{align*}
as \(R \to \infty\), which implies \eqref{eq: tightness in cond}.

Next, we investigate Condition \ref{cond: main2}. We say that a function \(\ell \colon \of 0, \infty \of\, \to \bR\) is \(\llambda\)-a.e. Skorokhod \(J_1\) continuous, if \(\omega^n \to \omega^0\) in the Skorokhod \(J_1\) topology implies that \(\ell (t, \omega^n) \to \ell (t, \omega^0)\) for \(\llambda\)-a.a. \(t \in \bR_+\).
From now on, we suppose that the coefficients \(\underline{b}, \overline{b}, \underline{a}, \overline{a}, \underline{\lambda}\) and \(\overline{\lambda}\) are \(\llambda\)-a.e. Skorokhod \(J_1\) continuous. This ensures that Condition~\ref{cond: main2} holds. To see this, take \(g \in C^2_b (\bR; \bR)\) and a sequence \((\omega^n)_{n = 0}^\infty \subset \Omega\) such that \(\omega^n \to \omega^0\) in the Skorokhod \(J_1\) topology. Then, for \(\llambda\)-a.a. \(t \in \bR_+\), the map
\[
(f, \omega) \mapsto \big| \mathcal{L} g (f, t, \omega, \omega (t)) - \mathcal{L} g (f, t, \omega^0 , \omega^0 (t)) \big|
\]
is continuous from \(F \times \{ \omega^n \colon n \in \mathbb{Z}_+\}\) into \(\bR_+\), where \( \{\omega^n \colon n \in \mathbb{Z}_+ \}\) is endowed with the Skorokhod \(J_1\) topology. Consequently, \eqref{eq: conv cond 2.8} follows from Berge's maximum theorem and the dominated convergence theorem. The assumption of \(\llambda\)-a.e. Skorokhod \(J_1\) continuity is rather mild. For instance, in Markovian cases it is implied by usual continuity assumptions, i.e., \((t, \omega) \mapsto \ell (\omega (t-))\) is \(\llambda\)-a.e. Skorokhod continuous once \(\ell \colon \bR \to \bR\) is continuous (cf. Remark \ref{rem: cond in markov} and \cite[VI.2.3]{JS}).

Finally, we examine Condition \ref{cond: main3}. We assume that the boundary functions are locally Lipschitz continuous, i.e., 
for every \(N > 0\) there exists a constant \(C = C_N > 0\) such that 
\begin{align*}
		|\underline{b}_t (\omega) - \underline{b}_t (\alpha)| + |\overline{b}_t (\omega) - \overline{b}_t (\alpha)| &\leq C  \sup_{s \in [0, t]} | \omega (s) - \alpha (s) |, \\
  | \underline{a}_t (\omega) - \underline{a}_t (\alpha)| + | \overline{a}_t (\omega) - \overline{a}_t (\alpha)| &\leq C  \sup_{s \in [0, t]} | \omega (s) - \alpha (s) |, \\
  | \underline{\lambda}_t (\omega) - \underline{\lambda}_t (\alpha)| + | \overline{\lambda}_t (\omega) - \overline{\lambda}_t (\alpha)| &\leq C  \sup_{s \in [0, t]} | \omega (s) - \alpha (s) |,
\end{align*}
for all \(\omega, \alpha \in \Omega \colon \sup_{s \in [0, N]} |\omega (s)| \vee |\alpha (s)| \leq N\) and \(t \in [0, N]\).
Then, Condition \ref{cond: main3} holds for instance in case of uniform ellipticity \(\underline{a} \geq 1/C\), 
or if the jump part is the sole source of randomness, i.e., \(\overline{a} = 0\).
To see this, note that, for \(z \neq 0\),
\begin{align*}
| k(f,t,\omega, z) -k(f,t, \alpha, z) | & = |\log(\lambda(f,t,\omega)) - \log(\lambda(f,t,\alpha)) |.
\end{align*}
This, together with the lower bound \(0 < \lambda_* \leq \lambda\), implies, for every \(N > 0\), the existence of a constant \(C_N > 0\) such that 
\begin{align*}
| k(f,t,\omega, z) -k(f,t, \alpha, z) | & \leq C_N  \sup_{s \in [0, t]} | \omega (s) - \alpha (s) |,
\end{align*}
for all \(\omega, \alpha \in \Omega \colon \sup_{s \in [0, N]} |\omega (s)| \vee |\alpha (s)| \leq N\) and \(t \in [0, N]\). In particular, we may choose \(\gamma = \gamma_N := C_N\).
Notice that \(u \mapsto \log(u)\) is uniformly bounded on \([\lambda_*, \lambda^*]\). Hence, there exists a constant \(K > 0\) such that 
\[
|k(f,t,\omega, z) | \leq K (1 + |z|) =: \beta (z), \quad (f,t,\omega, z) \in F \times \of 0, \infty\of \hspace{0.05cm}\times \hspace{0.05cm} \bR.
\]
As \( \int (1 + |z| + |z|^2)\, \n(dz) < \infty\), we conclude that Condition \ref{cond: main3} holds.

\section{The Relation to Control Rules} \label{sec: control}
Two classical concepts in stochastic optimal control theory are weak and relaxed control rules as, for instance, discussed in \cite{nicole1987compactification,EKNJ88, ElKa15}. In this section, we prove that our value function~\(v\) coincides with its weak and relaxed counterparts under Condition~\ref{cond: main1}. The control frameworks are particularly well-suited for weak convergence techniques.  As done for Markovian frameworks in \cite{nicole1987compactification}, we use such to establish regularity of the weak and relaxed value functions \(v^W\) and \(v^R\). These properties propagate directly to our value function \(v\). In this section we presume that \(F\) is a compact metrizable space.
\subsection{Weak and Relaxed Control Rules and their Value Functions}
 		Let \(\m (\bR_+ \times F)\) be the set of all Radon measures on \(\bR_+ \times F\) and define \(\m\) to be its subset of all measures in \(\m (\bR_+ \times F)\) whose projections on \(\bR_+\) coincide with the Lebesgue measure. 
   We endow \(\m\) with the vague topology, which turns it into a compact metrizable space (\cite[Theorem 2.2]{EKNJ88}). At this point, we remark that this topology coincides with the weak-strong topology (\cite{SPS_1981__15__529_0}) that is tested with compactly supported functions that are only continuous in the \(F\)-variable (instead of being jointly continuous), cf. \cite[Corollary~2.9]{SPS_1981__15__529_0}.
   Furthermore, we define 
		\[
		\m_0 := \big\{ m \in \m \colon m (ds, dq) = \delta_{\phi_s} (dq)ds \text{ for some Borel map } \phi \colon \bR_+ \to F \big\}.
		\]
		The identity map on \(\m\) is denoted by \(M\).
		We define the \(\sigma\)-field \[\M := \sigma (M_t (\phi); t \in \bR_+, \phi \in C_{c} (\bR_+ \times F; \bR)),\] where
		\[
		M_t (\phi) := \int_0^t \int \phi (s, f) M(ds, df).
		\]
 For \(g \in C^2_b (\bR^d; \bR)\), we define
		\[
		C (g) := g (X_{\cdot}) - \int_{0}^{\cdot} \int \mathcal{L} g (f, s, X, X_s) M (ds , df),
		\]
        where \(X\) denotes (with slight abuse of notation) the coordinate process on \(\Omega \times \m\), i.e., \(X (\omega, m) = \omega\) for \((\omega, m) \in \Omega \times \m\).
		A \emph{relaxed control rule} with initial value \((t, \omega) \in \of 0, \infty\of\) is a probability measure \(P\) on the product space \((\Omega \times \m, \mathcal{F} \otimes \M)\) such that \(P (X = \omega \text{ on } [0, t]) = 1\) and such that the processes \((C_s (g))_{s \geq t}\), with \(g \in C^2_b (\bR^d; \bR)\), are local \(P\)-martingales for the filtration
		\[
		\cG_s := \sigma (X_r, M_r (\phi); r \leq s, \phi \in C_c (\bR_+ \times F; \bR)), \quad s \in \bR_+.
		\]
		Finally, we define 
		\begin{align*}
		\mathcal{P}^R (t, \omega) &:= \big\{ \text{all relaxed control rules with initial value } (t, \omega) \big \}, \\
		\mathcal{P}^W (t, \omega) &:= \big\{ P \in \mathcal{P}^R (t, \omega) \colon P (M \in \m_0) = 1 \big\}, 
		\end{align*}
  and also the corresponding (relaxed and weak) value functions:
  \[
  v^R (t, \omega) := \sup_{P \in \mathcal{P}^R (t, \omega)} E^P \big[ \psi (X) \big], \qquad v^W (t, \omega) := \sup_{P \in \mathcal{P}^W (t, \omega)} E^P \big[ \psi (X) \big].
  \]
  Here, \(\psi\) is the same upper semianalytic function as in \eqref{eq: def vf}. The following theorem shows that, under certain conditions, our value function coincides with its relaxed and weak counterparts. It should be compared to \cite[Theorems~2.10, 8.7]{nicole1987compactification} where a related result was established for a jump diffusion setting. 
  
\begin{theorem} \label{theo: connection to control rules}
Assume that Condition \ref{cond: main1} holds. Then, for every \((t, \omega) \in \of 0, \infty\of\),
\[
\mathcal{C} (t, \omega) = \big\{ Q \circ X^{-1} \colon Q \in \mathcal{P}^R (t, \omega) \big\} = \big\{ Q \circ X^{-1} \colon Q \in \mathcal{P}^W (t, \omega) \big\}.
\]
In particular, \(
    v = v^R = v^W.
\)
\end{theorem}
	
\begin{remark}
		Theorem \ref{theo: connection to control rules} shows a direct connection between the concepts of nonlinear stochastic processes and the relaxed and weak formulations of stochastic optimal control problems. In particular, it connects the frameworks from the seminal papers \cite{ElKa15} and \cite{neufeld2017nonlinear}.
\end{remark}

\begin{remark}
We will show in Proposition \ref{prop: upper hemi} below that the set \(\cP^R(t, \omega)\) is compact for every \((t,\omega) \in \of 0, \infty\of \). 
Hence, by Theorem \ref{theo: connection to control rules}, the same holds true for \(\cC(t,\omega)\).
In particular, for every \(x \in \bR^d\), the set
\begin{align*}
    \mathfrak{P}_x(\Theta) &:= 
    \Big\{ P \in \mathfrak{P}_{\text{sem}}^{\text{ac}}\colon P \circ X_0^{-1} = \delta_x, (\llambda \otimes P)\text{-a.e. } (dB^P /d\llambda, dC^P/d\llambda, \nu^P/ d \llambda) \in \Theta (\,\cdot\,, X) \Big\}
     \\&=  \cC(0, \bx)
\end{align*}
is compact.
Such a compactness result has been obtained in \cite[Theorem 2.5]{neufeld}
and \cite[Theorem 4.41]{hol16} under global Lipschitz and boundedness assumptions, see also \cite{CN22a,CN22b} for such a result in a path-continuous framework.
As pointed out in \cite{neufeld}, a sufficient and in some cases even necessary condition for the set \(\mathfrak{P}_x(\Theta)\) to be closed is adequate control over the small jumps. More precisely, 
for the L\'evy case considered in \cite{neufeld}, where \(\Theta\) is independent of time and path, this condition can be phrased as
\begin{equation} \label{eq: small jumps}
    \lim_{\delta \to 0} \sup_{f \in F}  \int_{\|x\| \leq \delta} \|x\|^2  \, \n \o \bk(f, \cdot\,)^{-1}(dx) \to 0 \text{ as } \delta \to 0.
\end{equation}
Such a property is entailed in our framework. 
Recall that we equip the set of L\'evy measures~\(\mathcal{L}\) with the weak topology induced by the family of test functions \(\{ x \mapsto f(x) (\|x\|^2 \wedge 1) \colon f \in C_b(\bR^d;\bR) \}\). This topology is stronger compared to the weak topology induced by the collection of all bounded continuous functions vanishing in a neighborhood of the origin as used in \cite{hol16, neufeld}. Notice that, for the topology we use in this paper, any compact set \(K \subset \mathcal{L}\) has property \eqref{eq: small jumps}, i.e.,
\[
\lim_{\delta \to 0} \sup_{\nu \in K}  \int_{\|x\| \leq \delta} \|x\|^2  \, \nu(dx) \to 0 \text{ as } \delta \to 0.
\]
To see this, observe that the function
\[
\bR_+ \times \bR^d \ni (\delta, x) \mapsto \1_{\{\|x\| \leq \delta\}} 
\]
is upper semicontinuous.  Thus, by \cite[Theorem 8.10.61]{bogachev}, the map
\[
\bR_+ \times \mathcal{L} \ni (\delta, \nu) \mapsto \int_{\|x\| \leq \delta} (\|x\|^2 \wedge 1)  \, \nu(dx)
\]
is upper semicontinuous, too.
Hence, thanks to Berge's maximum theorem \cite[Lemma 17.30]{charalambos2013infinite}, the function
\[
g(\delta) := \sup_{\nu \in K} \, \int_{\|x\| \leq \delta} (\|x\|^2 \wedge 1)  \, \nu(dx)
\]
is again upper semicontinuous.
Therefore, 
\[
\limsup_{\delta \to 0} g(\delta) \leq g(0) = 0,
\]
as desired
\end{remark}

 \subsection{Proof of Theorem \ref{theo: connection to control rules}}
 Let us start with auxiliary technical observations. 
 \begin{lemma} \label{lem: convex combinations}
 	Let \(G \subset \bR^d \times \bR^{d \times d} \times \mathcal{L}\) be a compact convex set. Furthermore, let \(\mu\) be a Borel probability measure on \(\bR^d \times \bR^{d \times d} \times \mathcal{L}\) and let \(Y \colon \bR^d \times \bR^{d \times d} \times \mathcal{L} \to \bR^d \times \bR^{d \times d} \times \mathcal{L}\) be a Borel map. Then, 
 	\[
 	\mu (Y \in G) = 1 \quad \Longrightarrow \quad \int Y d \mu \in G.
 	\]
 \end{lemma}
\begin{proof}
Let \(\{\g_k \colon k \in \mathbb{N}\}\) be a countable set of continuous bounded functions that determine the convergence in  \(\mathcal{L}\), i.e., 
all functions \(g_k, \, k \in \mathbb{N},\) are integrable for all measures in \(\mathcal{L}\) and for every sequence \((\nu^n)_{n = 0}^\infty \subset \mathcal{L}\) we have \(\nu^n \to \nu^0\) in \(\mathcal{L}\) if and only if \(\int g_k d \nu^n \to \int g_k d \nu^0\) for all \(k \in \mathbb{N}\). The existence of such a {\em countable} family is well-known, see, e.g., \cite[Remark~VI.3.4]{JS}.
	
	Define a map \(\mathsf{J} \colon \bR^d \times \bR^{d \times d} \times \mathcal{L} \to \bR^d \times \bR^{d \times d} \times \mathbb{R}^{\mathbb{N}}\) by 
	\[
	\mathsf{J} (b, a, \nu) := \Big( b, a, \Big( \int \g_i d \nu\Big)_{i = 1}^\infty \Big).
	\]
	It is easy to see that \(\mathsf{J}\) is a linear homeomorphism. In particular, \(\mathsf{J}\) maps convex sets to convex sets.
	
	Let \(G\) and \(Y\) be as in the statement of the lemma and presume that \(\mu (Y \in G) = 1\). Define the push-forward measure \(\overline{\mu} := \mu \circ \mathsf{J} (Y)^{-1}\) and notice that \(\overline{\mu} ( \mathsf{J} (G) ) = 1\). As \(\mathsf{J}\) is a linear homeomorphism, the set \(\mathsf{J} (G)\) is compact and convex. Hence, it follows from \cite[Theorems~II.4.3 and II.6.2]{sion} that 
	\[
	\int \on{id} d \overline{\mu} \in \mathsf{J} (G).
	\]
	Using that
	\[
	\mathsf{J} \hspace{0.05cm}\Big( \int Y d \mu \Big) = \int \mathsf{J} (Y) d \mu = \int \on{id} d \overline{\mu} \in \mathsf{J} (G) \quad \Longleftrightarrow \quad \int Y d \mu \in G, 
	\]
	we conclude the claim of the lemma.
\end{proof}

For \(\omega, \alpha \in \Omega\) and \(t \in \bR_+\), we define the concatenation
\[
\omega \otimes_t \alpha := \omega \1_{[0, t)} + \alpha \1_{[t, \infty)}.
\]
\begin{lemma} \label{lem: control}
Suppose that Condition \ref{cond: main1} (i) and (iii) hold. 
    Take \((t, \omega) \in \of 0, \infty\of\) and let \(\mathscr{P}^t\) be the predictable \(\sigma\)-field for the canonical filtration generated by \(X_{\cdot + t}\). For every \(P \in \mathcal{C} (t, \omega)\), there exists a \(\mathscr{P}^t\)-measurable map \(\f = \f (P) \colon \of 0, \infty\of \hspace{0.05cm} \to F\) such that \((\llambda \otimes P)\)-a.e.
    \[
    (d B^P_{\cdot + t} / d\llambda, d C^P_{\cdot + t} / d \llambda, d \nu^P_{\cdot + t} / d \llambda) = (b, \sigma \sigma^*, \n \o k^{-1}) (\f, \cdot + t, \omega \otimes_t \cdot \hspace{0.05cm}).
    \]
\end{lemma}

\begin{proof}
Set
\[
G := \big\{ (s, \alpha) \in \of 0, \infty\of \colon (d B^P_{\cdot + t} / d\llambda, d C^P_{\cdot + t} / d \llambda, d \nu^P_{\cdot + t} / d \llambda) (s, \alpha) \in \Theta (s + t, \omega \otimes_t \alpha) \big\}.
\]
By \cite[Lemma 2.9]{CN22a}, the correspondence \((s, \alpha) \mapsto \Theta^* (s, \alpha) := \Theta (s + t, \omega \otimes_t \alpha)\) has a \(\mathscr{P}^t \otimes \mathcal{B}(\bR^d) \otimes \mathcal{B}(\bR^{d \times d}) \otimes \mathcal{B}(\mathcal{L})\)-measurable graph. Hence, 
\[
G = \big\{ (s, \alpha) \in \of 0, \infty \of \colon (s, \alpha, d B^P_{s + t} (\alpha) / d\llambda, d C^P_{s + t} (\alpha) / d \llambda, d \nu^P_{s + t} (\alpha) / d \llambda) \in \on{gr} \Theta^* \big\} \in \mathscr{P}^t.
\]
	For some arbitrary \(f_0 \in F\), set 
	\begin{align*}
\pi (s, \alpha) := \begin{cases} (d B^P_{\cdot + t} / d\llambda, d C^P_{\cdot + t} / d \llambda, d \nu^P_{\cdot + t} / d \llambda) (s, \alpha), & \text{ for } (s, \alpha) \in G, \\
(b, \sigma \sigma^*, \n \o k^{-1}) (f_0, s + t, \omega \otimes_t \alpha), & \text{ for \((s, \alpha) \not \in G\)}. \end{cases}
	\end{align*}
	The function \(\pi\) is \(\mathscr{P}^t\)-measurable. By Condition \ref{cond: main1} (iii), 
 \[
 (f, s, \alpha) \mapsto (b, \sigma \sigma^*, \n \o k^{-1}) (f, s + t, \omega \otimes_t \alpha)
 \]
 is a Carath\'eodory function in the sense that it is continuous in the \(F\)-variable and \(\mathscr{P}^t\)-measurable in the \(\of0, \infty\of\)-variable (where the \(\mathscr{P}^t\)-measurability of the third coordinate follows from \cite[Theorem 8.10.61]{bogachev}). Now, by Filippov's implicit function theorem (\cite[Theorem~18.17]{charalambos2013infinite}), there exists a \(\mathscr{P}^t\)-measurable function \(\f \colon \of 0, \infty\of \, \to F\) such that 
 \[
 \pi (s, \alpha) = (b, \sigma \sigma^*, \n \o k^{-1}) (\f (s, \alpha), s + t, \omega \otimes_t \alpha)
 \]
for all \((s, \alpha) \in \of 0, \infty\of\).
By the definition of \(\mathcal{C} (t, \omega)\), we have \((\llambda \otimes P)\)-a.e.
\[
\pi = (d B^P_{\cdot + t} / d\llambda, d C^P_{\cdot + t} / d \llambda, d \nu^P_{\cdot + t} / d \llambda).
\]
This completes the proof.
\end{proof}

\begin{lemma} \label{lem: v leq vW}
Suppose that Condition \ref{cond: main1} (i), (iii) and (iv) hold. Then, for every \((t, \omega) \in \of 0, \infty\of\),
\[
\cC (t, \omega) \subset \big\{ Q \circ X^{-1} \colon Q \in \mathcal{P}^W (t, \omega) \big\}.
\]
\end{lemma}
\begin{proof}
Take \((t, \omega) \in \of 0, \infty\of\) and \(P \in \cC (t, \omega)\). By Lemma~\ref{lem: control} and \cite[Theorem II.2.42]{JS}, there exists a \(\mathscr{P}^t\)-measurable map \(\f \colon \of 0, \infty\of \, \to F\) such that 
\begin{align*}
\K (g) &:= g (X_{\cdot + t}) - \int_0^\cdot \mathcal{L} g (\f (s, X), s + t, \omega \otimes_t X, X_{s + t})ds 
\\&\ = g (X_{\cdot + t}) - \int_{t}^{\cdot + t} \mathcal{L} g (\f (s - t, X), s, \omega \otimes_t X, X_s) ds.
\end{align*}
are local \(P\)-martingales for the canonical filtration of \(X_{\cdot + t}\) and all \(g \in C^2_b(\bR^d; \bR)\). Thanks to Condition \ref{cond: main1} (iv), the processes are bounded on compact time intervals. Using also that \(P\)-a.s. \(X = \omega\) on \([0, t]\), we get that all \(\K (g)\) are \(P\)-\((\cF_{s + t})_{s \geq 0}\)-martingales. Set 
\[
\mathfrak{M} (ds, df) := \delta_{\f ( [s - t] \, \vee\, 0,\, X)} (df) ds, 
\]
and \(Q := P \circ (X, \mathfrak{M})^{-1}\). It follows readily from the \(P\)-\((\cF_{s + t})_{s \geq 0}\)-martingale property of \(\K (g)\), that \((C_s (g))_{s \geq t}\) is a \(Q\)-\((\cG_s)_{s \geq 0}\)-martingale. Hence, we obtain \(Q \in \mathcal{P}^W (t, \omega)\) and consequently, 
\[
P = Q \circ X^{-1} \in \big\{ R \circ X^{-1} \colon R \in \mathcal{P}^W (t, \omega) \big\}.
\]
The proof is complete.
\end{proof}

\begin{lemma} \label{lem: vR leq v}
Suppose that Condition \ref{cond: main1} holds. Then, for every \((t, \omega) \in \of 0, \infty\of\),
\[
\big\{ Q \circ X^{-1} \colon Q \in \mathcal{P}^R (t, \omega) \big\} \subset \cC (t, \omega).
\]
\end{lemma}
\begin{proof}
Take \((t, \omega) \in \of 0, \infty\of\) and \(Q \in \mathcal{P}^R (t, \omega)\). 
By \cite[Lemma 3.2]{lackerSPA}, there exists a \((\mathcal{G}_t)_{t \geq 0}\)-predictable probability kernel \(\mathsf{k} \colon \bR_+ \times  \m \times \mathcal{B}(F) \to \mathbb{R}_+\) such that 
\[
m (ds, df) = \mathsf{k} (s, m, df)  ds.
\]
For \((s, \alpha, m) \in \of 0, \infty\of \hspace{0.05cm} \times \hspace{0.05cm} \m\), define 
\[
\pi (s, \alpha, m) := \int_F ( b, \sigma \sigma^*, \n \o k^{-1}) (f, s, \alpha) \mathsf{k} (s, m, df).
\]
Thanks to Condition \ref{cond: main1} (i) -- (iii) and Lemma \ref{lem: convex combinations}, \(\pi (s, \alpha, m) \in \Theta (s, \alpha)\) for all \((s, \alpha, m) \in \of 0, \infty\of \hspace{0.05cm} \times \hspace{0.05cm} \m\). By Filippov's implicit functions theorem (\cite[Theorem 18.17]{charalambos2013infinite}), there exists a predictable function \(\mathsf{f} \colon \of 0, \infty\of \hspace{0.05cm} \times \hspace{0.05cm} \m \to F\) such that 
\[
\pi (s, \alpha, m) = ( b, \sigma \sigma^*, \n \o k^{-1})( \mathsf{f} (s, \alpha, m), s, \alpha), \quad (s, \alpha, m) \in \of 0, \infty\of \hspace{0.05cm} \times \hspace{0.05cm} \m.
\]
We conclude that, for all \((\alpha, m) \in \Omega \times \m\),
\[
\int_F (b, \sigma \sigma^*, \n \o k^{-1})(f, s, \alpha) m (ds, df) = (b, \sigma \sigma^*, \n \o k^{-1})(\mathsf{f} (s, \alpha, m), s, \alpha) ds.
\]
In particular, for every \(g \in C^2_b (\bR^d; \bR)\),
\[
\int_0^\cdot \mathcal{L} g (\mathsf{f} (s, X, M), s, X, X_s) ds = \int_0^\cdot \int_F \mathcal{L} g (f, s, X, X_s) M(ds, df).
\]
Hence, for every \(g \in C^2_b (\bR^d; \bR)\), as \(Q \in \mathcal{P}^R (t, \omega)\), and using Condition \ref{cond: main1} (iv),
the process 
\[
g (X_s) - \int_0^s \mathcal{L} g (\mathsf{f} (r, X, M), r, X, X_r) dr, \quad \ s \geq t,
\]
is a \(Q\)-\((\cG_s)_{s \geq 0}\)-martingale. 
Let \((\cF^*_s)_{s \geq 0}\) be the (right-continuous) filtration generated by~\(X_{\cdot + t}\). By \cite[Theorem~9.19, Proposition~9.24; (9.23)]{jacod79}, the process 
\[
g (X_{\cdot + t}) - \int_0^{\cdot} E^Q \big[ \mathcal{L} g (\mathsf{f} (r + t, X, M), r + t, X, X_{(r + t)-}) \big | \mathcal{F}^*_{r-} \big] dr
\]
is a (local) \(Q\)-\((\cF^*_s)_{s \geq 0}\)-martingale. Thanks to \cite[Theorem II.2.42]{JS}, this implies that \(X_{\cdot + t}\) is a \(Q\)-\((\cF^*_s)_{s \geq 0}\)-semimartingale with characteristics 
\begin{align*}
	B^Q &= \int_0^\cdot E^Q \big[ b (\mathsf{f} (r + t, X, M), r + t, X) \big| \cF^*_{r-} \big] dr, \\
	C^Q &= \int_0^\cdot E^Q \big[ \sigma \sigma^* (\mathsf{f} (r + t, X, M), r + t, X) \big| \cF^*_{r-} \big] dr, \\
	\nu^Q (dr, G) &= E^Q \Big[ \int \1_{G \backslash \{0\}} (k (\mathsf{f} (r + t, X, M), r + t, X, y)) \n (dy) \Big| \cF^*_{r-} \Big] dr,
\end{align*}
where \(G \in \mathcal{B}(\bR^d)\). Thanks to Condition \ref{cond: main1} (i) -- (iii) and Lemma \ref{lem: convex combinations}, we conclude that \((\llambda \otimes Q)\)-a.e. \((d B^Q / d \llambda, d C^Q / d \llambda, d \nu^Q / d \llambda) \in \Theta (\cdot + t, X)\). Consequently, \cite[Lemma 2.9]{jacod80} yields that \(Q \circ X^{-1} \in \cC (t, \omega)\). This completes the proof.
\end{proof}

\begin{proof}[Proof of Theorem \ref{theo: connection to control rules}]
Notice that 
\[
\big\{ Q \circ X^{-1} \colon Q \in \mathcal{P}^W (t, \omega) \big\} \subset \big\{ Q \circ X^{-1} \colon Q \in \mathcal{P}^R (t, \omega) \big\}.
\]
Hence, the claimed equality follows from Lemmata~\ref{lem: v leq vW} and \ref{lem: vR leq v}.
\end{proof}

\section{Proofs for the Regularity Results} \label{sec: pf regularity}

In this section we prove our main Theorems \ref{theo: value function upper semi} and \ref{theo: value function lower semi} on the regularity of the value function \(v\). Our proofs rely on the relation of \(v\) to its control counterparts \(v^R\) and \(v^W\). We show in Section \ref{sec: vW upper semi} that \(v^R\) is upper semicontinuous and in Section \ref{sec: vR lower semi} that \(v^W\) is lower semicontinuous. By Theorem~\ref{theo: connection to control rules}, these regularity properties transfer to our value function~\(v\).

Let us outline the argument. We argue in the spirit of Berge's maximum theorem (\cite[Lemma~17.29, Lemma~17.30, Theorem~17.31]{charalambos2013infinite}).
More precisely, we show in Proposition \ref{prop: PR closed} that the correspondence \( \cP^R\) has closed graph. This is a key step in order to verify the sequential compactness property of \(\cP^R\) in Proposition \ref{prop: upper hemi}. 
By \cite[Theorem~17.20]{charalambos2013infinite}, this is equivalent to \(\cP^R\) being upper hemicontinuous with compact values. Then, the proof of Theorem \ref{theo: value function upper semi} can be seen as a stochastic version of Berge's maximum theorem, where upper semicontinuity of the input function can be relaxed to \(\cC(t,\omega)\)-a.s. upper semicontinuity by virtue of the continuous mapping theorem. To get the closed graph, we employ martingale problem techniques. This methodology is inspired by the approach from the seminal paper~\cite{nicole1987compactification}. 

Similarly, we prove in Proposition \ref{prop: PW lower hemi} below that the correspondence \(\cP^W\) is lower hemicontinuous, see \cite[Theorem 17.21]{charalambos2013infinite}. Using again the continuous mapping theorem, this gives access to the lower semicontinuity of the value function for \(\cC(t,\omega)\)-a.s. lower semicontinuous input functions.

\subsection{Proof of upper semicontinuity of the value functions: Theorem \ref{theo: value function upper semi}} \label{sec: vW upper semi}
 We start with two preliminary results before we present the proof of Theorem~\ref{theo: value function upper semi}.

\begin{proposition} \label{prop: PR closed}
	Suppose that Condition \ref{cond: main2} holds and take \(T > 0\).
	Let \((t^n, \omega^n)_{n = 0}^\infty \subset [0, T] \times D([0, T]; \bR^d)\) be a sequence such that \((t^n, \omega^n) \to (t^0, \omega^0)\).\footnote{Recall that we endow \([0, T] \times D([0, T]; \bR^d)\) with the topology induced by the pseudometric \(\mathsf{d}\) (and that we identify elements \((t, \omega), (s, \alpha)\) such that \(\mathsf{d} ((t, \omega), (s, \alpha)) = 0\), which turns \(\mathsf{d}\) into a metric).}
	Then, for every sequence \((P^n)_{n = 0}^\infty\) of probability measures on \((\Omega \times \m, \cF \otimes \mathcal{M})\) such that \(P^n \in \mathcal{P}^R (t^n, \omega^n)\) for \(n \in \mathbb{N}\) and \(P^n \to P^0\) weakly, it holds that \(P^0 \in \mathcal{P}^R (t^0, \omega^0)\).
\end{proposition}
\begin{proof}
	It is well-known (\cite[Lemma 7.7, p. 131]{EK}) that there exists a set \(D \subset \bR_+\) with countable complement such that \(X_s\) is \(P^0\)-a.s. continuous for all \(s \in D\). Consequently, by the continuous mapping theorem and the right-continuity of \(X\), \(P^0\)-a.s. \(X_s = \omega^0 (s)\) for all \(s \in [0, t^0)\). Next, take a function \(g \in C^2_b (\bR^d; \bR)\) and define 
	\begin{align*}
		\K^n_R := g (X_R) - g (\omega^n (t^n)) - \int_{t^n}^R \int_F \mathcal{L} g (f, s, X, X_s) M(ds, df)
	\end{align*}
	for \(n \in \mathbb{Z}_+\) and \(R \in D\) with \(R \geq t^n\). By Taylor's theorem, and the definition of a truncation function,
	we have, for all \(x, y \in \bR^d\),
	\[
	\big| g (x + y) - g(x) - \langle \nabla g (x), h (y) \rangle \big| \leq C \|\nabla^2 g \|_\infty (\|y\|^2 \wedge 1),
	\]
	where the constant \(C > 0\) depends on \(g\) and the truncation function \(h\).
	Consequently, using the first part of Condition \ref{cond: main1}, we obtain
	\begin{align} \label{eq: conv Mn}
		|\K^n_R - \K^0_R| \leq C \big[ | g (\omega^n (t^n)) - g (\omega^0 (t^0)) | + |t^n - t^0| \big], 
	\end{align}
	where \(C > 0\) is a deterministic constant that only depends on \(g\) and the bounds for the coefficients provided by Condition \ref{cond: main1}. 
 Next, we show that the function \((\alpha, m) \mapsto \K^0_R (\alpha, m)\) is \(P^0\)-a.s. continuous. It suffices to prove continuity of 
 \[
 (\alpha, m) \mapsto \int_{t^0}^R \int_F \mathcal{L} g (f, s, \alpha, \alpha (s)) m(ds, df) .
 \]
 In this regard, consider a sequence \((\alpha^n, m^n)_{n = 0}^\infty \subset \Omega  \times \m \) such that \((\alpha^n, m^n) \to (\alpha^0, m^0)\). We obtain
 \begin{align*}
 \Big| \int_{t^0}^R & \int_F \mathcal{L} g (f, s, \alpha^0, \alpha^0(s)) m^0(ds, df) - \int_{t^0}^R \int_F \mathcal{L} g (f, s, \alpha^n, \alpha^n(s)) m^n(ds, df)\Big| 
 \\&\leq  \Big| \int_{t^0}^R \int_F \mathcal{L} g (f, s, \alpha^0, \alpha^0(s)) m^0(ds, df) - \int_{t^0}^R \int_F \mathcal{L} g (f, s, \alpha^0, \alpha^0(s)) m^n(ds, df)\Big| \\& \qquad +  \Big| \int_{t^0}^R \int_F \mathcal{L} g (f, s, \alpha^0, \alpha^0(s)) m^n(ds, df) - \int_{t^0}^R \int_F \mathcal{L} g (f, s, \alpha^n, \alpha^n(s)) m^n(ds, df)\Big| 
 \\&\leq  \Big| \int_{t^0}^R \int_F \mathcal{L} g (f, s, \alpha^0, \alpha^0(s)) m^0(ds, df) - \int_{t^0}^R \int_F \mathcal{L} g (f, s, \alpha^0, \alpha^0(s)) m^n(ds, df)\Big| \\&\qquad+   \int_{t^0}^R \int_F \big| \mathcal{L} g (f, s, \alpha^0, \alpha^0(s))- \mathcal{L} g (f, s, \alpha^n, \alpha^n(s)) \big| m^n(ds, df)
  \\&\leq  \Big| \int_{t^0}^R \int_F \mathcal{L} g (f, s, \alpha^0, \alpha^0(s)) m^0(ds, df) - \int_{t^0}^R \int_F \mathcal{L} g (f, s, \alpha^0, \alpha^0(s)) m^n(ds, df)\Big| 
  \\&\qquad+   \int_{t^0}^R \sup_{f \in F} \big| \mathcal{L} g (f, s, \alpha^0, \alpha^0(s)) -  \mathcal{L} g (f, s, \alpha^n, \alpha^n(s)) \big| ds
 \\&=: I_n + II_n.
 \end{align*}
 For every \(s \in [0, R]\), the map \(f \mapsto \mathcal{L} g (f, s, X^0, X^0_s)\) is continuous by Condition~\ref{cond: main1}~(iii). Thus, \(I_n \to 0\) follows from Condition \ref{cond: main1}~(iv) and \cite[Corollary~2.9]{SPS_1981__15__529_0}. Using Condition~\ref{cond: main2}, we obtain \(II_n \to 0\). Hence, \((\alpha, m) \mapsto \K^0_T (\alpha, m)\) is \(P^0\)-a.s. continuous.
 In summary, using the continuous mapping theorem and \eqref{eq: conv Mn}, we obtain, for every bounded \(P^0\)-a.s. continuous function \(\g\) on \(\Omega \times \m\), that
	\[
	E^{P^0} \big[ \K^0_T \, \g \big] = \lim_{n \to \infty} E^{P^n} \big[ \K^0_T \, \g \big] = \lim_{n \to \infty} E^{P^n} \big[ \K^n_T\, \g \big].
	\]
	Consequently, for every \(t^0 < S < R\) with \(S, R \in D\), and every \(\cG_S\)-measurable bounded \(P^0\)-a.s. continuous \(\g\), we get 
	\begin{align*}
		E^{P^0}\big[ \big(\K^0_R - \K^0_S\big)\, \g \big] = \lim_{n \to \infty} E^{P^n} \big[ \big(\K^n_R - \K^n_S\big)\, \g\big] = 0.
	\end{align*}
	We conclude that \(P^0\)-a.s. \(E^{P^0} [ \K^0_R | \cG_S ] = \K^0_S\). As \(D^c\) is countable, the right-continuity of \(\K^0\) and the L\'evy--Doob downward theorem (\cite[Theorem II.51.1]{RW1}) yield, for all \(s \in [t^0, R)\), that \(P^0\)-a.s.
	\[
	E^{P^0} \big[ \K^0_R | \cG_{s+} \big] = \K^0_s,
	\]
	and, by the tower rule, \(P^0\)-a.s.
	\[
	E^{P^0} \big[ \K^0_R | \cG_{s} \big] = \K^0_s.
	\]
	Using again that \(s \mapsto \K^0_s\) is right-continuous and, in addition, (uniformly) bounded on compact subsets of \([t^0, \infty)\), it follows from the density of \(D\) in \(\bR_+\), and the dominated convergence theorem for conditional expectations, that \(P^0\)-a.s.
	\[
	E^{P^0} \big[ \K^0_{t} | \cG_{s} \big] = \K^0_s, 
	\]
	for \(t^0 \leq s < t\). This means that \((\K^0_s)_{s \geq t^0}\) is a \(P^0\)-martingale for the filtration \((\cG_t)_{t \geq 0}\). Finally, notice that, for every \(R \in D\) with \(R > t^0\), 
	\begin{align*}
	E^{P^0}\big[ g (X_{t^0})\big] - g (\omega^0 (t^0)) &= E^{P^0} \big[ \K^0_{t^0}\big] = E^{P^0} \big[ \K^0_{R}\big] \\&= \lim_{n \to \infty} E^{P^n} \big[ \K^n_R \big] = \lim_{n \to \infty} E^{P^n} \big[ \K^n_{t^n} \big] = 0.
	\end{align*}
	As this holds for all \(g \in C^2_b (\bR^d; \bR)\), we conclude that \(P^0 (X_{t^0} = \omega^0 (t^0)) = 1\). This completes the proof of \(P^0 \in \mathcal{P}^R (t^0, \omega^0)\) and therefore, of the proposition.
\end{proof}

\begin{remark}
    As the proof of Proposition \ref{prop: PR closed} shows, the second part of Condition \ref{cond: main2} can be weakened to the following assumption: for every \(t \in \bR_+\), \(g \in C^2_b (\bR^d; \bR)\), and every compact subset \(\mathscr{M} \subset \m\) the family of functions
\[
\Omega \ni \omega \mapsto \int_0^t \int_F \mathcal{L} g (f, s, \omega, \omega (s)) m(ds,df), \quad m \in \mathscr{M},
\]
is equicontinuous.
\end{remark}

\begin{proposition} \label{prop: upper hemi}
Suppose that Condition \ref{cond: main2} holds and take \(T > 0\). Let \((t^n, \omega^n)_{n = 0}^\infty \subset [0, T] \times D([0, T]; \bR^d)\) be a sequence such that \((t^n, \omega^n) \to (t^0, \omega^0)\). Then, every sequence \((P^n)_{n = 1}^\infty\) with \(P^n \in \mathcal{P}^R (t^n, \omega^n)\) has an accumulation point in \(\mathcal{P}^R (t^0, \omega^0)\).
\end{proposition}
\begin{proof}
As \(\omega^1, \omega^2, \dots\) are arbitrary elements of \(\Omega\), the measures \(P^1, P^2, \dots\) are not necessarily laws of {\em global} semimartingales (this can only be the case when \(\omega^1, \omega^2, \dots\) are of finite variation on \([0, t_1], [0, t_2], \dots\)). Hence, we cannot directly apply tightness criteria for semimartingale laws.
We resolve this issue as follows. First, we cut off the initial values and prove tightness for this modified sequence.
In a second step, we deduce tightness of the original sequence by continuously adding the initial values.

\smallskip
\emph{Step 1}. We set \(Y^n := X_{\cdot \vee t^n}\) and \(\mathbb{B}^n := (\Omega \times \m, \mathcal{F} \otimes \mathcal{M}, (\cG_t)_{t \geq 0}, P^n)\).
By \cite[Theorem~II.2.42]{JS}, \(Y^n\) is a semimartingale on the stochastic basis \(\mathbb{B}^n\) with starting value \(Y_0 = \omega^n (t^n)\) and characteristics 
\begin{align*}
B^{P^n} &= \int_{t^n}^{\cdot \vee t^n}\int b (f, s, X) M (ds, df), \\
C^{P^n} &= \int_{t^n}^{\cdot \vee t^n} \int \sigma \sigma^* (f, s, X) M (ds, df), \\
\nu^{P^n} ( [0, \cdot\,] \times G) &= \int_{t^n}^{\cdot \vee t^n}\int \1_{G \backslash \{0\}} ( k (f, s, X, y)) \n (dy) M (ds, df), \quad G \in \mathcal{B}(\bR^d).
\end{align*}
In the following, we want to apply \cite[Theorem VI.5.10]{JS} to prove that the family \(\{Q^n := P^n \circ (Y^n)^{-1} \colon n \in \mathbb{N}\}\) is tight. We check its prerequisites. For \(N > 0\), the final part of Condition~\ref{cond: main1} (iv) yields that
\begin{equation} \label{eq: JS (ii)}
\begin{split}
\sup_{n \in \mathbb{N}} P^n (\nu^{P^n} &([0, N] \times \{x \in \bR^d \colon \|x\| > R\}) > \varepsilon) 
\\&\leq \tfrac{N}{\varepsilon} \sup \big\{ \n (\|k (f, s, \alpha, \cdot ) \| > R ) \colon (f, s, \alpha) \in F \times \of 0, T\gs \big\} \to 0 \text{ as } R \to \infty.
\end{split}
\end{equation}
Furthermore, using Condition \ref{cond: main1} (iii) and (iv), for every \(T > 0\), we obtain the existence of a constant \(C = C_T > 0\) such that, for all \(n \in \mathbb{N}\) and \(s \in [0, T]\), 
\begin{align} \label{eq: JS (iii)+(iv)}
\sum_{k = 1}^d \big[\on{Var}_s ( B^{P^n, k} ) + C^{P^n, kk}_s \big] + \int (\|x\|^2 \wedge 1) \nu^{P^n} ([0, s] \times dx) \leq C s,
\end{align}
where \(\on{Var}(\cdot)\) denotes the variation process. Moreover, notice that 
\begin{align} \label{eq: initial value bounded}
\text{the set \(\{\omega^n (t^n) \colon n \in \mathbb{N}\}\) is bounded},
\end{align}
since \(\omega^n (t^n) \to \omega^0 (t^0)\) by hypothesis. 
Thanks to \eqref{eq: JS (ii)} - \eqref{eq: initial value bounded}, we deduce from \cite[Theorem~VI.5.10, Remark~VI.5.13]{JS} that \(\{Q^n \colon n \in \mathbb{N}\}\) is tight.

{\em Step 2.} 
Define \(\alpha^n := \omega^n (\cdot \wedge t^n) - \omega^n (t^n)\) and \(\overline{Q}^n := P^n \circ (\alpha^n, Y^n, M)^{-1}\), where the latter are considered to be probability measures on \((\Omega \times \Omega \times \m, \cF \otimes \cF \otimes \mathcal{M})\). Here, we use the product topology for \(\Omega \times \Omega \times \m\). At this point, we stress that \(\Omega \times \Omega \ni (\omega, \alpha) \mapsto \omega + \alpha\) is not continuous (cf. \cite[Remark~VI.1.22]{JS}). In the remainder of this proof, we establish an almost sure continuity property that follows from results in \cite{whitt}. 
Notice that \(\alpha^n \to \alpha^0\) uniformly. 
Using this observation, Step~1 and the compactness of \(\m\), we obtain the existence of a weakly convergent subsequence \((\overline{Q}^{N_n})_{n = 1}^\infty \subset (\overline{Q}^n)_{n = 1}^\infty\) whose limit we denote by \(\overline{Q}^0\). Clearly, \(\overline{Q}^0 (d \omega \times \Omega \times \m) = \delta_{\alpha^0} (d \omega)\), which implies that the first two coordinates are independent under \(\overline{Q}^0\). Next, we explain that the second coordinate has no fixed times of discontinuity under \(\overline{Q}^0\). We write \(X\) for the projection to the second coordinate. Take \(S, \varepsilon > 0\). There exists a set \(D \subset \bR_+\), with countable complement, such that \(X_t\) is \(\overline{Q}^0\)-a.s. continuous for every \(t \in D\) (\cite[Lemma~7.7, p. 131]{EK}). Take \(s, t \in D\) such that \(s < S < t\). Thanks to \cite[Proposition~VI.2.4]{JS}, the map \(\sup_{s < r \leq t} \|\Delta X_r\|\) is \(\overline{Q}^0\)-a.s. continuous. Hence, by the continuous mapping and the Portmanteau theorem, we obtain
\begin{align*}
\overline{Q}^0 \big( \|\Delta X_S\| > \varepsilon \big) &\leq \overline{Q}^0 \Big( \sup_{s < r \leq t} \|\Delta X_r\| > \varepsilon\Big)
\\&\leq \liminf_{n \to \infty} \overline{Q}^{N_n} \Big( \sup_{s < r \leq t} \|\Delta X_r\| > \varepsilon \Big) 
\\&\leq \liminf_{n \to \infty} \overline{Q}^{N_n} \Big( \sum_{s < r \leq t} \1_{\{\|\Delta X_r\| > \varepsilon\}} \geq 1 \Big) 
\\&\leq \liminf_{n \to \infty} E^{\overline{Q}^{N_n}} \Big[ \sum_{s < r \leq t} \1_{\{\|\Delta X_r\| > \varepsilon\}} \Big]
\\&= \liminf_{n \to \infty} E^{\overline{Q}^{N_n}} \big[ \nu^{P^{N_n}} ((s, t] \times \{x \in \bR^d \colon \|x\| > \varepsilon\}) \big]
\\&\leq \frac{(t - s)}{\varepsilon^2 \wedge 1} \sup \Big\{ \int (\|k (f, r, \omega, z)\|^2 \wedge 1) \n (dz) \colon (f, r, \omega) \in F \times \of 0, N\gs\Big\}, 
\end{align*}
where \(N > t\). Notice that the last term is finite by Condition \ref{cond: main1} (iv).
As \(D^c\) is countable, we can take \(s \nearrow S\) and \(t \searrow S\) and obtain that 
\[
\overline{Q}^0 \big( \|\Delta X_S\| > \varepsilon \big)  = 0.
\]
We deduce from \cite[Theorem~4.1, Lemma~4.3]{whitt} that the map \[\Omega \times \Omega \times \m \ni (\alpha, \omega, m) \mapsto \mathfrak{A} (\alpha, \omega, m) := (\alpha + \omega, m)\] is \(\overline{Q}^0\)-a.s. continuous. Consequently, the continuous mapping theorem yields that 
\begin{align*}
	\overline{Q}^{N_n} \circ \mathfrak{A}^{-1} \to \overline{Q}^{0} \circ \mathfrak{A}^{-1}.
\end{align*}
Notice that \(\overline{Q}^{N_n} \circ \mathfrak{A}^{-1} = P^{N_n}\) and therefore, that \(\overline{Q}^{0} \circ \mathfrak{A}^{-1} \in \mathcal{P}^R (t^0, \omega^0)\) by Proposition \ref{prop: PR closed}. 
This completes the proof.
\end{proof}

\begin{proof}[Proof of Theorem \ref{theo: value function upper semi}]
Take \((t^n, \omega^n)_{n = 0}^\infty \subset [0, T] \times D([0, T]; \bR^d)\) such that \((t^n,\omega^n) \to (t^0, \omega^0)\) and let \(\psi \colon \Omega \to \bR\) be a bounded \(\cC (t^0, \omega^0)\)-a.s. upper semicontinuous function. 
Recall from Theorem~\ref{theo: connection to control rules} that 
\[
\cC (t^0, \omega^0) = \big\{ Q \circ X^{-1} \colon Q \in \mathcal{P}^R (t^0, \omega^0) \big\}.
\]
Hence, as \(X\) is continuous on \(\Omega \times \m\), \(\psi \circ X\) is \(\mathcal{P}^R (t^0, \omega^0)\)-a.s. upper semicontinuous. 
Pick a subsequence \((N_n)_{n = 1}^\infty\) such that 
\[
\limsup_{n \to \infty} v^R (t^n, \omega^n) = \lim_{n \to \infty} v^R (t^{N_n}, \omega^{N_n}).
\]
Further, for every \(n \in \mathbb{N}\), there exists a measure \(Q^{N_n} \in \mathcal{P}^R (t^{N_n}, \omega^{N_n})\) such that 
\[
v^R (t^{N_n}, \omega^{N_n}) - \frac{1}{N_n} \leq E^{Q^{N_n}} \big[ \psi (X) \big] \leq v^R (t^{N_n}, \omega^{N_n}).
\]
Hence, 
\[
\limsup_{n \to \infty} v^R (t^n, \omega^n) = \lim_{n \to \infty} E^{Q^{N_n}} \big[ \psi (X) \big].
\]
By Proposition \ref{prop: upper hemi}, passing to a subsequence if necessary, there exists a measure 
\(Q^0 \in \mathcal{P}^R (t^0, \omega^0)\)
such that \(Q^{N_n} \to Q^0\) weakly.
Consequently, by a version of the continuous mapping theorem (\cite[Example 17, p. 73]{pollard}), as \(\psi \circ X\) is \(Q^0\)-a.s. upper semicontinuous, 
\[
\limsup_{n \to \infty} v^R (t^n, \omega^n) = \lim_{n \to \infty} E^{Q^{N_n}} \big[ \psi (X) \big] \leq E^{Q^0} \big[ \psi (X)\big] \leq v^R (t^0, \omega^0).
\]
This shows that \(v^R\) is upper semicontinuous at \((t^0, \omega^0)\). As \(v = v^R\) by Theorem \ref{theo: connection to control rules}, the proof is complete. 
\end{proof}

\subsection{Proof of lower semicontinuity of the value function: Theorem \ref{theo: value function lower semi}} \label{sec: vR lower semi}

The key observation is given by the following. 
\begin{proposition} \label{prop: PW lower hemi}
Suppose that the Conditions \ref{cond: main1} and \ref{cond: main3} hold. Then, for every \(T > 0\), every sequence \((\omega^n, t^n)_{n = 0}^\infty \subset [0, T]
	\times D([0, T]; \bR^d)\) such that \((t^n, \omega^n) \to (t^0, \omega^0)\) and every measure \(Q^0 \in \mathcal{P}^W (t^0, \omega^0)\), there exists a sequence \((Q^n)_{n = 1}^\infty\) such that \(Q^n \in \mathcal{P}^W (t^n, \omega^n)\) and \(Q^n \to Q^0\) weakly.
\end{proposition}

\begin{proof}
	{\em Step 1.} By \cite[Lemma 3.2]{lackerSPA}, there exists a \((\mathcal{G}_t)_{t \geq 0}\)-predictable probability kernel \(\mathsf{k} \colon \bR_+ \times  \m \times \mathcal{B}(F) \to \mathbb{R}_+\) such that 
\[
m (ds, df) = \mathsf{k} (s, m, df)  ds.
\]
	We set \(G := \{ \delta_x \colon x \in F\}\) and, for some arbitrary \(f_0 \in F\), 
	\[
	\pi (t,  m, df) := \begin{cases}\delta_{f_0} (df), & \text{if } \mathsf{k} (t, m, df) \not \in G,\\ \mathsf{k} (t, m, df),& \text{if } \mathsf{k} (t, m, df) \in G, \end{cases}
	\]
	for \((t, m) \in \bR_+ \times \mathbb{M}\). By the continuity of the map \(x \mapsto \delta_x\) and the compactness of \(F\), we deduce from \cite[Theorem 18.17]{charalambos2013infinite} that there exists a \((\mathcal{G}_t)_{t \geq 0}\)-predictable map \((t, m) \mapsto \gamma_t (m)\) such that
	\[
	\pi (t, m, df) = \delta_{\gamma_t (m)} (df)
	\]
	for all \((t, m) \in \bR_+ \times \m\).
	As \(Q^0 (M \in \m_0) = 1\), it follows that \((\llambda \otimes Q^0)\)-a.e. \(\mathsf{k} = \delta_\gamma\).
 
 \smallskip
	{\em Step 2.} Thanks to \cite[Theorem II.2.42]{JS} and \cite[Theorem 14.68]{jacod79}, on an extension \(\mathbb{B}^0\) of the stochastic basis \((\Omega \times \m, \cF \otimes \mathcal{M}, (\cG_s)_{s \geq 0}, Q^0)\), there exists an \(r\)-dimensional standard Brownian motion \(W\) and a Poisson random measure \(\mathfrak{p}\) with intensity measure \(\mathfrak{q} := \llambda \otimes \n\) such that a.s.
	\[
	\begin{cases} X = \omega^0 & \text{ on } [0, t^0], \\
		d X_s = b (\gamma_s, s, X) d s + \sigma (\gamma_s, s, X) d W_s 
		\\\hspace{2cm}+ \int h (k (\gamma_s, s, X, y) (\mathfrak{p} - \mathfrak{q}) (ds, dy) 
		\\\hspace{2cm}+ \int \big(k(\gamma_s, s, X, y) - h(k (\gamma_s, s, X, y)) \big) \mathfrak{p} (ds, dy), & \text{ on } (t^0, \infty). \end{cases}
	\]
	Using the local Lipschitz conditions and boundedness conditions from the Conditions~\ref{cond: main1} (iv) and \ref{cond: main3}, we deduce from \cite[Theorem 14.23]{jacod79} that there exists a \cadlag adapted process \(X^n\) on the stochastic basis \(\mathbb{B}^0\) such that a.s.
		\[
	\begin{cases} X^n = \omega^n & \text{ on } [0, t^n], \\
		d X^n_s = b (\gamma_s, s, X^n) d s + \sigma (\gamma_s, s, X^n) d W_s 
		\\\hspace{2cm}+ \int h (k (\gamma_s, s, X^n, y) (\mathfrak{p} - \mathfrak{q}) (ds, dy) 
		\\\hspace{2cm}+ \int \big(k(\gamma_s, s, X^n, y) - h(k (\gamma_s, s, X^n, y)) \big) \mathfrak{p} (ds, dy), & \text{ on } (t^n, \infty). \end{cases}
	\]
	Let \(Q^n\) be the law of \((X^n, \delta_\gamma \,d \llambda)\), seen as a probability measure on \((\Omega \times \m, \cF \otimes \mathcal{M})\). Then, by \cite[Theorem II.2.42]{JS}, \(Q^n \in \mathcal{P}^W (t^n, \omega^n)\). In the following, we prove that \(X^n \to X^0\) in the ucp topology as \(n \to \infty\), i.e., we prove that \(X^n \to X^0\) uniformly on compact time intervals in probability. Once this is done, Theorem \ref{theo: value function lower semi} follows, as then \(Q^n \to Q^0\) weakly.
   	In this regard, take \(N > T\) and \(\varepsilon > 0\), and define 
	\[
	T_\h^n := \inf \{s \geq 0 \colon \|X^n_s\| \vee \|X^0_s\| \geq \h \text{ or } \|X^n_{s-}\| \vee \|X^0_{s-}\| \geq \h \}
	\]
	for \(\h > 0\).
 Let \(P\) and \(E\) be the probability measure and the expectation on \(\mathbb{B}^0\).
	Using the Burkholder--Davis--Gundy inequality and the boundedness assumptions on \(b, \sigma\) and \(v\) from Conditions~\ref{cond: main1}~(iv) and \ref{cond: main3}, it follows readily that 
	\[
	\sup_{n \in \mathbb{Z}_+} E \Big[ \sup_{s \in [0, N]} \|X^n_s\|^2 \Big] < \infty.
	\]
	Chebyshev's inequality yields that
	\[
	P (T^n_\h \leq N) = P \Big( \sup_{s \in [0, N]} \|X^n_s\| \vee \|X^0_s\| \geq \h \Big) \leq \frac{2}{\h^2} \sup_{n \in \mathbb{Z}_+} E \Big[ \sup_{s \in [0, N]} \|X^n_s\|^2 \Big].
	\]
	Hence, there exists an \(\h = \h (\varepsilon) > 0\) such that 
	\[
	P (T^n_\h \leq N) \leq \varepsilon.
	\]
	Set \(s^n := t^n \wedge t^0\) and \(S^n := t^n \vee t^0\), and notice that 
	\begin{equation} \label{eq: three terms ucp}
		\begin{split}
	E \Big[ \sup_{s \in [0, N \wedge T^n_\h]} &\|X^n_s - X^0_s\|^2 \Big] \\&\leq 2 E \Big[ \sup_{s \in [0, S^n \wedge T^n_\h]} \| X^n_s - X^0_s\|^2 \Big] \\&\qquad + 2 E \Big[ \sup_{s \in [S^n \wedge T^n_\h, N \wedge T^n_\h]} \|X^n_s - X^n_{S_n \wedge T^n_\h} - (X^0_s - X^0_{S_n \wedge T^n_\h})\|^2 \Big].
	\end{split}
\end{equation}
Using the boundedness assumptions on \(b, \sigma\) and \(v\) from Conditions \ref{cond: main1} (iv) and \ref{cond: main3}, and the Burkholder--Davis--Gundy inequality, we obtain that
\begin{align} \label{eq: first term ineq}
E \Big[ \sup_{s \in [0, S^n \wedge T^n_\h]} \|X^n_s - X^0_s\|^2 \Big] \leq C\, \Big(\sup_{s \in [0, T]} \|\omega^n (s \wedge t^n) - \omega^0 (s \wedge t^0)\|^2 + S^n - s^n\Big),
\end{align}
where the constant does not depend on \(n\). 
Similarly, using the local Lipschitz conditions from Condition \ref{cond: main3}, we obtain that 
\begin{equation} \label{eq: main gronwall ineq}
	\begin{split}
	E \Big[ \sup_{s \in [S^n \wedge T^n_\h, N \wedge T^n_\h]} \|X^n_s - X^n_{S_n \wedge T^n_\h} &- (X^0_s - X^0_{S_n \wedge T^n_\h})\|^2 \Big] \\&\leq C \int_0^N E \Big[ \sup_{r \in [0, s \wedge T^n_\h]} \|X^n_r - X^0_r\|^2 \Big] ds,
\end{split}
\end{equation}
where \(C = C(N) > 0\) does not depend on \(n\) and \(N \mapsto C(N)\) is increasing. Finally, we deduce from \eqref{eq: three terms ucp} -- \eqref{eq: main gronwall ineq} and Gronwall's lemma that 
\[
E \Big[ \sup_{s \in [0, N \wedge T^n_\h]} \|X^n_s - X^0_s\|^2 \Big] \leq C \Big(\sup_{s \in [0, T]} \|\omega^n (s \wedge t^n) - \omega^0 (s \wedge t^0)\|^2 + S^n - s^n\Big) \to 0.
\]
Therefore, for every \(\delta > 0\) there exists a \(K > 0\) such that 
\begin{align*}
	P \Big( \sup_{s \in [0, N]} \|X^n_s - X^0_s\| \geq \delta\Big) &\leq P \Big( \sup_{s \in [0, N]} \|X^n_s - X^0_s\| \geq \delta, T^n_\h > N\Big) + \varepsilon
	\\&\leq \frac{1}{\delta^2} E \Big[ \sup_{s \in [0, N \wedge T^n_\h]} \|X^n_s - X^0_s\|^2 \Big] + \varepsilon 
	\leq 2 \varepsilon
\end{align*}
for all \(n \geq K\). This means that \(X^n \to X^0\) in the ucp topology. 
The proof is complete.
\end{proof}

\begin{proof}[Proof of Theorem \ref{theo: value function lower semi}]
Take \((t^n, \omega^n)_{n = 0}^\infty \subset [0, T] \times D([0, T]; \bR^d)\) such that \((t^n,\omega^n) \to (t^0, \omega^0)\) and let \(\psi \colon \Omega \to \bR\) be a bounded \(\cC (t^0, \omega^0)\)-a.s. lower semicontinuous function. 
Thanks to Theorem~\ref{theo: connection to control rules}, \(\psi \circ X\) is \(\mathcal{P}^W (t^0, \omega^0)\)-a.s. lower semicontinuous.
For every \(\varepsilon > 0\), there exists a measure \(Q = Q(\varepsilon) \in \mathcal{P}^W (t^0, \omega^0)\) such that 
\[
v^W (t^0, \omega^0) \leq E^{Q} \big[ \psi (X) \big] + \varepsilon. 
\]
By Proposition \ref{prop: PW lower hemi}, there exists a sequence \((Q^n)_{n = 1}^\infty\) such that \(Q^n \in \mathcal{P}^W (t^n, \omega^n)\) and \(Q^n \to Q\) weakly. By a version of the continuous mapping theorem (\cite[Example 17, p. 73]{pollard}), as \(\psi \circ X\) is \(Q\)-a.s. lower semicontinuous, we obtain that 
\begin{align*}
    \liminf_{n \to \infty} v^W (t^n, \omega^n) \geq \liminf_{n \to \infty} E^{Q^n} \big[ \psi (X) \big] \geq E^{Q} \big[ \psi (X) \big] \geq v^W (t^0, \omega^0) - \varepsilon.
\end{align*}
As \(\varepsilon > 0\) is arbitrary, we conclude that \(v^W\) is lower semicontinuous at \((t^0, \omega^0)\). Since \(v = v^W\) by Theorem \ref{theo: connection to control rules}, the proof is complete.
\end{proof}

\section{Proofs for the Markovian Case} \label{sec: pf markovian}

In the following, we will frequently use results from \cite{K19}. To this end, let us define the \emph{symbol} associated to the coefficients \((\bb, \bs, \n \o \bk^{-1})\) by 
\begin{equation} \label{eq: df symbol}
\begin{split}
q(f, x, \xi)  := -i \langle \bb(f,x), \xi \rangle & + \frac{1}{2} \langle \xi, \bs \bs^*(f,x) \xi \rangle \\ &+ 
\int \big(1 - e^{i \langle y, \xi \rangle} + i \langle \xi, h(y) \rangle \big) (\n \o \bk(f,x, \cdot)^{-1})(dy),
\end{split}
\end{equation}
for \(x, \xi \in \bR^d\) and \(f \in F\). Here, \(i\) denotes the imaginary number.

The next result emerges as a special case of \cite[Lemma A.2]{K19} by setting \(I:= F \times \bR^d\). We will make use of it throughout this section.

\begin{lemma} \label{lem: symbol}
    Suppose that Condition \ref{cond: main1 markov} holds.
    Then, 
    \[
    \lim_{r \to 0} \sup_{f \in F} \sup_{x \in \bR^d} \sup_{\| \xi \| \leq r} | q(f,x,\xi) | = 0.
    \]
\end{lemma}

\subsection{Feller properties: Proof of Theorem \ref{thm: feller}} \label{sec: pf feller}
We focus on the \(\usc_b\)--Feller property. The \(C_b\)--Feller property can be established in similar fashion. 
Let \(t \in \bR_+\). As \(S_0 = \on{id}\), we may assume \(t > 0\).
Consider an upper semicontinuous function \(\psi \colon \bR^d \to \bR\).
It follows from Remark \ref{rem: as continuity} that \(\psi(X_t)\) is \(\cC(s, \omega)\)-a.s. upper semicontinuous for every \(\omega \in \Omega\) and \(s \leq t\). Hence, we deduce from Theorem \ref{theo: value function upper semi} that the map 
\[
[0,t] \times D([0,t]; \bR^d) \ni (s, \omega) \mapsto \sup_{P \in \cC(s, \omega)} E^P[ \psi(X_t)]
\]
is upper semicontinuous (where \([0, t] \times D([0, t]; \bR^d)\) is endowed with the topology induced by~\(\d\)). As the embedding
\[
\bR^d \ni x \mapsto (0, \bx) \in [0,t]  \times D([0,t]; \bR^d)
\]
is continuous, we deduce upper semicontinuity of 
\(\bR^d \ni x \mapsto \cE^x(\psi(X_t)) = S_t(\psi)(x) \).
This completes the proof.\qed

\subsection{Form of Generators: Proof of Proposition \ref{prop: generator}} \label{sec: pf generator}
We verify the prerequisites of \cite[Lemma 6.1]{K19}. This will imply the claim.
Note first that, for every \(r > 0\),
\[
 \sup_{f \in F} \sup_{\|x\| \leq r}    \Big(    \|\bb (f, x)\| + \|\bs (f,x)\| + \int \, (\|\bk (f, x, z)\|^2 \wedge 1)\, \n (dz) \Big) < \infty,
\]
thanks to part (iv) of Condition \ref{cond: main1 markov}.
Further, for every \(f \in F\) and \(x \in \bR^d\), there exists a measure \(P \in \cC(0, \bx)\) with
\[
 (dB^P_s /d\llambda, dC^P_s/d\llambda, \nu^P_s/ d \llambda) 
 = (\bb(f, X_{s-}), \bs \bs^*(f, X_{s-}), \n \o \bk(f, X_{s-}, \cdot\, )^{-1})
\]
for \((\llambda \otimes P)\text{-a.e. }\) \((s,\omega) \in \of 0,  \infty \of\).
Indeed, this follows from \cite[Corollary~IX.2.33]{JS} once we checked its prerequisites.
In view of Condition \ref{cond: main1 markov}, it suffices to prove continuity of 
\[
x \mapsto  \int  h^i(\bk(f,x,z)) h^j(\bk(f,x,z)) \, \n(dz), \quad
x \mapsto \int g(\bk(f,x,z)) \, \n(dz),
\]
for \(i,j =1, \dots, d\) and every continuous bounded function \(g \colon \bR^d \to \bR\) vanishing around the origin.
As the truncation function \(h\) is chosen to be (Lipschitz) continuous, the functions
\[
y \mapsto \frac{h^i(y)h^j(y)}{\|y\|^2 \wedge 1} \, \1_{\{y \not = 0\}}, \quad y \mapsto \frac{g(y)}{\|y\|^2 \wedge 1}
\]
are \(\n \o \bk^{-1}\)-a.e. continuous and bounded. Hence, the desired continuity follows from part  (iii) of Condition \ref{cond: main1 markov} and the continuous mapping theorem.

\smallskip
Next, we check the continuity requirements of \cite[Lemma 6.1]{K19}.

\emph{Step 1}. For every compact subset \(K \subset \bR^d\), the families
\[
K \ni x \mapsto \bb(f,x), \quad K \ni x \mapsto \bs \bs^*(f,x)\colon\quad f \in F,
\]
are uniformly equicontinuous.
To see this, part (iii) of Condition \ref{cond: main1 markov} shows that the map
\[
F \times \bR^d \ni (f, x ) \mapsto (\bb(f,x), \bs \bs^*(f,x))
\]
is continuous, and in particular uniformly continuous on the compact set \(F \times K\). This implies the claim.

\emph{Step 2}. For every compact subset \(K \subset \bR^d\) and \(g \in C^1_b(\bR^d; \bR)\) with \(|g(y)| \leq \|y\|^2 \wedge 1\) for all \(y \in \bR^d\), the family
\[
K \ni x \mapsto \int g(\bk(f,x,z))\ \n (dz)\colon \quad f \in F,
\]
is uniformly equicontinuous.
Notice that the function
\[
y \mapsto \frac{g(y)}{\|y\|^2 \wedge 1} \1_{\{y \not = 0\}}
\]
is \(\n \o \bk^{-1}\)-a.e. continuous and bounded.
Hence, by part (iii) of Condition \ref{cond: main1 markov} and the continuous mapping theorem, the function
\[
F \times \bR^d \ni (f,x) \mapsto \int  g(\bk(f,x,z)) \, \n(dz)
\]
is continuous, and in particular uniformly continuous on the compact set \(F \times K\). This implies the claim.

\smallskip
\emph{Step 3}. The uniform continuity property
\[
\lim_{r \to \infty} \sup_{\| y- x \| \leq r} \sup_{\| \xi \| \leq r^{-1}} \sup_{f \in F} \, |q(f, y, \xi) | = 0
\]
holds for every \(x \in \bR^d\).
Indeed, this follows from Lemma \ref{lem: symbol}. \qed

\subsection{Viscosity Solution: Proofs of Theorems \ref{thm: viscosity} and \ref{thm: viscosity uniqueness}} \label{sec: pf viscosity}
We start with three auxiliary results.

\begin{proposition} \label{prop: continuity in time}
    Suppose that Condition \ref{cond: main1 markov} holds.
    Let
    \((S_t)_{t \in \bR_+}\) be the family of operators
    defined in \eqref{eq: semigroup},
    take \(\psi \in C_b(\bR^d; \bR)\) and let \(K \subset \bR^d\) be compact.
    Then, 
    \[
    \bR_+ \ni t \mapsto S_t(\psi)(x)
    \]
    is continuous uniformly in \(x \in K\).
\end{proposition}
\begin{proof}
    Thanks to part (iv) of Condition \ref{cond: main1 markov}, Theorem \ref{thm: feller} and Lemma \ref{lem: symbol}, this follows from \cite[Theorem 5.3 (ii)]{K19}.
\end{proof}

\begin{lemma} \label{lem: G cont}
Suppose that Condition \ref{cond: main1 markov} holds.
Let \(\phi \in C_b^\infty(\bR^d;\bR)\). Then, the function
\[
\bR^d \ni x \mapsto G(x,\phi)
\]
is continuous and bounded.
\end{lemma}
\begin{proof}
Notice that, for all \(x, y \in \bR^d\),
	\[
	\big| \phi (x + y) - \phi(x) - \langle \nabla \phi (x), h (y) \rangle \big| \leq C \|\nabla^2 \phi \|_\infty (\|y\|^2 \wedge 1),
	\]
where the constant \(C > 0\) depends on \(\phi\) and the truncation function \(h\).
Hence, parts (iii) and (iv) of Condition \ref{cond: main1 markov} and the dominated convergence theorem imply that 
\[
F \times \bR^d \ni (f,x) \mapsto \int\, \big[ \phi(x + \bk (f, x, z)) - \phi(x) - \langle \nabla \phi(x) , h (\bk (f, x, z)) \rangle\big]\, \n (dz)
\]
is continuous.
Using parts (i), (iii) and (iv) of Condition \ref{cond: main1 markov}, and Berge's maximum theorem (\cite[Lemma 17.30]{charalambos2013infinite}), we conclude that \( x \mapsto G(x,\phi)\) is continuous and bounded.   
\end{proof}

\begin{lemma} \label{lem: viscosity}
Suppose that Condition \ref{cond: main1 markov} holds.
    Let \((T_t)_{t \in \bR_+}\) be a sublinear Markovian semigroup on \(C_b(\bR^d; \bR)\) with the following properties:
\begin{enumerate}
    \item[\textup{(i)}] for \(\psi \in C_b(\bR^d; \bR)\), the map \((t,x) \mapsto T_t(\psi)(x)\) is continuous;
    \item[\textup{(ii)}]  for all \(x \in \bR^d\) and \(\phi \in C_c^\infty(\bR^d; \bR)\),
    \[
    \lim_{t \to 0} \frac{T_t (\phi) (x) - \phi (x)}{t} = G(x, \phi).
    \]
\end{enumerate}
 Then, \((t,x) \mapsto T_t(\psi)(x)\) is a viscosity solution to \eqref{eq: PDE G}.
\end{lemma}
\begin{proof}
Let \(A\) be the (pointwise) generator of \((T_t)_{t \in \bR_+}\).
 Thanks to part (iv) of Condition \ref{cond: main1 markov} and \cite[Corollary 4.7]{K21}, we obtain that
\[
    A (\phi )(x) = G(x, \phi), \quad  \phi \in C_b^\infty(\bR^d; \bR),\, x \in \bR^d.
\]
Using Lemma \ref{lem: G cont}, it follows that \(x \mapsto G (x, \phi)\) is continuous and bounded and consequently, we conclude that \(C_b^\infty (\bR^d; \bR) \subset \cD(A)\).
Combining this observation with Proposition~\ref{prop: weak sense}, we deduce that \((t,x) \mapsto T_t(\psi)(x)\) is a (bounded) viscosity solution to \eqref{eq: PDE G}.   
\end{proof}

\begin{proof}[Proof of Theorem \ref{thm: viscosity}]
By Proposition \ref{prop: generator} and Lemma~\ref{lem: viscosity}, it suffices to show that 
 \( (t,x) \mapsto S_t(\psi)(x)\) is continuous. 
Consider a sequence \((t^n, x^n)_{n = 0}^\infty\) in \(\bR_+ \times \bR^d\) such that \((t^n, x^n) \to (t^0, x^0)\) as \(n \to \infty\).
Notice that, for \(n \in \bN\),
\begin{equation} \label{eq: pf jointly cont}
    | S_{t^0} (\psi)(x^0) - S_{t^n} (\psi)(x^n) | \leq | S_{t^0} (\psi)(x^0) - S_{t^0} (\psi) (x^n) | + | S_{t^0}(\psi) (x^n)  - S_{t^n}(\psi)(x^n) |.
\end{equation}
As \(x \mapsto S_{t^0} (\psi)(x) \) is continuous by Theorem \ref{thm: feller}, the first term on the right hand side of~\eqref{eq: pf jointly cont} vanishes as \(n \to \infty\), while the second term goes to zero by Proposition \ref{prop: continuity in time}.
This completes the proof.
\end{proof}

\begin{proof}[Proof of Theorem \ref{thm: viscosity uniqueness}]
By virtue of Theorem \ref{thm: viscosity}, \((t,x) \mapsto S_t(\psi)(x)\) is a viscosity solution to \eqref{eq: PDE G}. Hence, the claim follows from Theorem~\ref{theo: comparison} in the Appendix.
\end{proof}
%

\appendix

\section{A Comparison Result for L\'evy-type HJB Equations}
The purpose of this appendix is to tailor the abstract comparison result given by \cite[Corollary~2.23]{hol16} to our framework. 

\begin{condition} \label{cond: local bdd}
	The functions \(\bb\) and \(\bs\) satisfy the following boundedness condition:
	\[
	\forall x \in \bR^d \colon \quad \sup \Big\{ \|\bb (f, x)\| + \|\bs(f, x)\| \colon f \in F \Big\} < \infty.
	\]
	Moreover, there is a constant \(C > 0\) and a Borel function \(\gamma \colon L \to [0, \infty]\) such that \[\int (\gamma^2(z) \wedge \gamma(z))\, \n (dz) < \infty\] and 
	\begin{align*}
		\|\bb(f,x)- \bb(f,y) \| + \| \bs(f,x)-\bs(f,y) \| &\leq C  \|x -y \|,  \\
		\| \bk (f, x, z) - \bk (f, y, z)\| &\leq  \gamma(z)  \|x - y\|, \\
		\| \bk (f, x, z)\| & \leq  \gamma(z) 
	\end{align*}
	for all \(f \in F\), \(x,y \in \bR^d\) and \(z \in L\).
\end{condition}

\begin{theorem} \label{theo: comparison}
Suppose that Condition~\ref{cond: local bdd} holds, that \(u \colon \bR_+ \times \bR^d \to \bR\) is a bounded viscosity subsolution and that \(v \colon \bR_+ \times \bR^d \to \bR\) is a bounded viscosity supersolution to the nonlinear PDE \eqref{eq: PDE G}. Then, \(u \leq v\).
\end{theorem}

\begin{proof}
	Using arguments as in \cite[Lemma~2.32, Proposition~2.33]{hol16}, the result can be deduced from \cite[Corollary~2.23]{hol16}. For brevity, we provide not all details. Instead, we sketch the idea that relates our setting to the HJB framework from \cite[Definition~2.27]{hol16}. Let \(\gamma \colon L \to [0,\infty]\) be as in Condition~\ref{cond: local bdd}. For \(\phi \in C^2_b (\bR^d; \bR)\), set 
	\begin{align*}
		J^f (x) := \int\, \big[ \phi(x + \bk (f, x, z)) - \phi(x) - \langle \nabla \phi(x) , h (\bk (f, x, z)) \rangle\big]\, \n (dz).
	\end{align*}
	Taking some \(\kappa \in (0, 1)\), we split the operator \(J^f\) as follows: 
	\begin{align*}
		J^f (x) &= H^f (x) + \langle \nabla \phi (x), K^f (x) \rangle, \phantom \int 
	\end{align*}
with 
\begin{align*}
H^f (x) &:= \int_{\{\gamma \leq \kappa\}} \big[ \phi(x + \bk (f, x, z)) - \phi(x) - \langle \nabla \phi(x) , \bk (f, x, z) \rangle\big]\, \n (dz) 
		\\&\qquad \qquad + \int_{\{ \kappa < \gamma \leq 1\}} \big[ \phi(x + \bk (f, x, z)) - \phi(x) - \langle \nabla \phi(x) , \bk (f, x, z) \rangle\big]\, \n (dz)
		\\& \qquad \qquad + \int_{\{ \gamma > 1\}} \big[ \phi(x + \bk (f, x, z)) - \phi(x) \big]\, \n (dz),
  \\
K^f (x) &:=  \int_{\{\gamma \leq 1\}} \big(\bk (f, x, z) - h (\bk (f, x, z)) \big) \, \n (dz) - \int_{\{\gamma > 1\}}  h (\bk (f, x, z)) \, \n (dz).
\end{align*}
At this point, we notice that \(K^f\) is well-defined. 
Indeed, regarding the first term, observe that, for \(z \in \{\gamma \leq 1\}\) and all \(f \in F, x \in \bR^d\),
\[
\| \bk (f, x, z) - h (\bk (f, x, z))\| \leq C \|\bk (f, x, z)\| \1_{\{\|\bk (f, x, z)\|\, > \,\varepsilon\}} \leq C\, \|\bk (f, x, z)\|^2 \leq C ( \gamma^2 (z) \wedge 1),
\]
where \(\varepsilon > 0\) is such that \(h (y) = y\) for \(\|y\| \leq \varepsilon\).\footnote{Such an \(\varepsilon\) exists by definition of a truncation function.} Similarly, for the second term, we record that
\[
\|h (\bk(f,x,z)\| \1_{\{\gamma(z)\, >\, 1\}} \leq C\, ( \gamma^2(z) \wedge 1 ), \quad (f,x,z) \in F \times \bR^d \times L,
\]
where the constant \(C > 0\) bounds the truncation function \(h\).
This shows that \(K^f\) is well-defined under Condition~\ref{cond: local bdd}. Returning to the decomposition \(J^f = H^f + \langle \nabla \phi, K^f\rangle\), we notice that \(H^f\) is of a similar form as the HJB operator from \cite[Definition~2.27]{hol16} with the exception that \(\n\) is not defined on \(\mathcal{B}(\bR^d)\) but on \(\mathcal{B}(L)\), and the decomposition is not done with the sets \(\{|z| \leq \kappa\}, \{\kappa < |z| \leq 1\}\) and \(\{|z| > 1\}\) but with \(\{\gamma \leq \kappa\}, \{\kappa < \gamma \leq 1\}\) and \(\{\gamma > 1\}\). A careful inspection of the proofs for \cite[Lemma~2.32, Proposition~2.33]{hol16} shows that the argument for the validity of the prerequisites of \cite[Corollary~2.23]{hol16} only needs cosmetic modifications (namely, \(|z|\) in the estimates from \cite{hol16} has to be replaced by \(\gamma (z)\) in our setting). The term \(K^f\) can be added to the drift coefficient \(\bb\) and therefore be treated as in \cite{hol16}. 
Here, we emphasise that, under Condition~\ref{cond: local bdd}, the function \(x \mapsto K^f(x)\) is Lipschitz continuous, uniformly in \(f \in F\). To see this, notice that 
\begin{equation} \label{eq: small gamma weg}
\begin{split}
            \int_{\{\gamma \leq 1\}} \big( \bk (f, x, z) &- h (\bk (f, x, z)) \big)\, \n (f, dz) 
            \\&= \int_{\{\epsilon < \gamma \leq 1\}} \big( \bk (f, x, z) - h (\bk (f, x, z)) \big)\, \n (f, dz),
\end{split}
\end{equation}
for some \(\epsilon > 0\) that satisfies \(h(y) = y\) for all \(\|y\| \leq \epsilon\), since the assumption \(\|\bk\| \leq \gamma\) yields \(\gamma > \epsilon\) in case \(\|\bk\| > \epsilon\). This readily implies the desired Lipschitz continuity of  \(K^f\).
Omitting for brevity full details of the argument sketched above, we conclude the claim of Theorem~\ref{theo: comparison} from \cite[Corollary~2.23]{hol16}.
\end{proof}

\end{document}